\documentclass[a4paper,10pt]{article}
\usepackage[english]{babel}
\usepackage[utf8]{inputenc}
\usepackage[T1]{fontenc}
\usepackage{textcomp} 
\usepackage{amsmath}
\usepackage{amsmath}
\usepackage{enumerate}
\usepackage{amsfonts}
\usepackage{amsthm}
\usepackage{mathrsfs}
\usepackage{amssymb}
\usepackage{comment}
\usepackage{csquotes}
\usepackage{bm}
\usepackage{graphicx}
\usepackage{bbm}
\usepackage{array}
\usepackage{ragged2e}
\usepackage{mathtools}
\usepackage{fancybox,framed}
\usepackage{xcolor}
\usepackage{lipsum}
\usepackage{tabularx}
\usepackage{changepage}
\usepackage[mathscr]{euscript}
\DeclareSymbolFont{rsfs}{U}{rsfs}{m}{n}
\DeclareSymbolFontAlphabet{\mathscrsfs}{rsfs}

\theoremstyle{definition}
\newtheorem{Def}{Definition}[section]

\newtheorem{Rmk}[Def]{Remark}

\theoremstyle{plain}
\newtheorem{Prop}[Def]{Proposition}
\newtheorem{Thm}[Def]{Theorem}
\newtheorem{Lemma}[Def]{Lemma}
\newtheorem{Cor}[Def]{Corollary}

\newcommand{\footremember}[2]{%
   \footnote{#2}
    \newcounter{#1}
    \setcounter{#1}{\value{footnote}}%
}

\newcommand{\R}{\mathbb{R}}

\renewcommand{\P}{\mathbb{P}}
\newcommand{\E}{\mathbb{E}}

\newcommand{\N}{\mathbb{N}}

\newcommand\HHH{\mathfrak{H}}
\newcommand\D{\mathbb{D}}

\newcommand{\norm}[1]{\left\|#1\right\|}

\newcolumntype{M}[1]{>{\raggedright}m{#1}}

\renewcommand{\epsilon}{\varepsilon}

\DeclareMathOperator*{\Dom}{Dom}

\title{Malliavin calculus and densities for chaos-driven stochastic differential equations}
\author{Laurent Loosveldt\footremember{1}{Université de Liège, Belgium}\qquad Yassine Nachit\footremember{2}{ Université de Lille, France} \qquad Ivan Nourdin\footremember{3}{Université du Luxembourg, Luxembourg}}
\begin{document}

\maketitle

\begin{abstract}
We study stochastic differential equations driven by finite-order chaos processes on abstract Wiener spaces, with pathwise Riemann–Stieltjes integration. The driving noise is an $\mathbb{R}^m$-valued chaotic process given by multiple Wiener-Itô integrals of fixed order, allowing for non-Gaussian dynamics. Under mild smoothness assumptions on the coefficients and Hölder-type regularity of the noise, we establish existence and uniqueness of solutions. We then prove Malliavin differentiability and absolute continuity of the law of the solution. Since the usual Gaussian isonormal framework is unavailable, we rely on the Kusuoka-Stroock approach to Malliavin calculus and develop a Taylor expansion for multiple integrals under Cameron-Martin shifts. Under suitable ellipticity, independence, and non-degeneracy conditions, the Bouleau–Hirsch criterion yields density results. Applications to multidimensional Hermite-driven equations are provided.
\end{abstract}



\section{Introduction}
This paper is concerned with the stochastic differential equation
\begin{equation}\label{eq:SDEsDrivenF}
    X^k_t=x_0^k+ \int_{0}^{t}b_k(X_s)ds + \sum_{\ell=1}^{m}\int_{0}^{t}\sigma_{k,\ell}(X_s)dF_s^{\ell},
  \end{equation}
studied in the following general framework:
\begin{itemize}
\item $x_0\in \R^d$ is the initial value of the process $X$;
\item the stochastic integral in \eqref{eq:SDEsDrivenF} is understood pathwise, in the Riemann--Stieltjes sense (see e.g.\ \cite{young1936inequality});
\item the driving process $F$ is an $\R^m$-valued chaos process defined on an abstract Wiener space $(\Omega,\mathcal{F},\P,\mathfrak{H})$ (see definition (\ref{def:stochasticintegrals}) below);
\item the coefficients $(b_k)_{1 \leq k \leq d}$ and $(\sigma_{k,\ell})_{1 \leq k \leq d, 1 \leq \ell \leq m}$ are (smooth) deterministic functions, as made precise by hypothesis $\mathbf{(H_1)}$ in Subsection \ref{subsec:existencesol}.
\end{itemize}

The question of the existence of a density for the solution to \eqref{eq:SDEsDrivenF} in the classical case where $F$ is a multidimensional Brownian motion was, in fact, Malliavin's original motivation in \cite{MR0536013} for introducing what is now known as Malliavin calculus. Over the last two decades, a substantial body of work has been devoted to the case where $F=(B_t^{H,1},\dots,B_t^{H,d})_{t \geq 0}$ is an $m$-dimensional fractional Brownian motion of Hurst parameter $H \in (0,1)$, that is, a centered Gaussian process with continuous trajectories and covariance
\begin{equation}\label{eqn:covFBM}
\E[B_s^{H,\ell}B_t^{H,\ell'}]=\delta_{\ell,\ell'} \frac{1}{2}(|s|^{2H}+|t|^{2H}-|t-s|^{2H}),
\end{equation}
see e.g.\ \cite{Baudoin2007,MR2680405,Hairer2007,Hairer2011,Hairer2013,Neuenkirch2007,MR1893308,nualart2009malliavin}. Among these contributions, \cite{nualart2009malliavin} investigates the Malliavin smoothness of the solution to \eqref{eq:SDEsDrivenF} and, as an application, establishes that this solution admits a density with respect to the Lebesgue measure. The results obtained in \cite{nualart2009malliavin} subsequently turned out to be key ingredients in several further developments: smoothness of the density under a Hörmander-type hypoellipticity condition \cite{Baudoin2007}, upper bounds for the density \cite{Baudoin2014}, asymptotic expansions for expectations of smooth functionals of the solution \cite{Neuenkirch2009}, rates of convergence and asymptotic error distribution of Euler approximation schemes \cite{Hu2016}, and statistical inference \cite{Chronopoulou2013}. The strategy introduced in \cite{nualart2009malliavin} has also been extended in various directions, covering for instance mixed stochastic differential equations \cite{Shevchenko2013}, delay equations \cite{Leon2011} and fractional stochastic Volterra integral equations \cite{Besalu2021}. While \cite{nualart2009malliavin} only deals with the range $H \in (1/2,1)$, the authors of \cite{MR2680405} reinterpret equation \eqref{eq:SDEsDrivenF} in the rough-paths sense of \cite{Lyons1998}, thereby covering much more general rough Gaussian processes, and in particular the case $H \in (1/4,1/2)$.

The purpose of the present paper is to push this theory further, in the general context of chaotic processes on abstract Wiener spaces. More precisely, we show that, under mild conditions on the coefficients $(b_k)_{1 \leq k \leq d}$ and $(\sigma_{k,\ell})_{1 \leq k \leq d, 1 \leq \ell \leq m}$ (hypothesis $\mathbf{(H_1)}$ in Subsection \ref{subsec:existencesol}) and on the driving process $F$ (hypothesis $\mathbf{(H_2)}$ in Subsection \ref{subsec:existencesol}), the stochastic differential equation \eqref{eq:SDEsDrivenF} admits a unique solution. Under additional assumptions on the coefficients (hypothesis $\mathbf{(H_3)}$ in Section \ref{sec:density}) and under independence and non-degeneracy assumptions (hypotheses $\mathbf{(H_4)}$ and $\mathbf{(H_5)}$ in Section \ref{sec:density}), which naturally extend their standard counterparts in the Gaussian setting, we then prove that this solution admits a density.

Let us stress at once a central difficulty in this last task. In order to establish the absolute continuity of the law of the solution $\{X_t\}_{t \in [0,T]}$, we rely on the Bouleau--Hirsch criterion (Theorem \ref{thm:BouleauHirsch} below), which is formulated in terms of Malliavin calculus. This strategy has already proved its effectiveness in the Gaussian settings recalled above. There, however, the analysis is considerably simplified by the fact that Malliavin calculus can be set up via an isonormal process canonically associated with the underlying Gaussian process, which greatly facilitates the computation of the Malliavin derivatives of the quantities involved. This favourable picture breaks down as soon as the driving process $F$ is no longer Gaussian. Consequently, the classical methods developed in the Gaussian setting are not directly available to us, and new strategies have to be devised. Our main tool is an alternative definition of Malliavin calculus, which is presented in Subsection \ref{sub:malialter}.

The paper is organised as follows. Section \ref{sec:backgroundstrat} is devoted to the background of our work: we first recall the definition of abstract Wiener spaces, together with Malliavin calculus and multiple integrals in this framework; we then specify conditions under which the deterministic equation associated with \eqref{eq:SDEsDrivenF} admits a solution, and derive from them a condition on the driving process $F$ ensuring the existence of a solution to the stochastic differential equation \eqref{eq:SDEsDrivenF}. Section \ref{sec:backgroundstrat} ends with a detailed presentation of the strategies we will use to establish the Malliavin smoothness of this solution. In Section \ref{sec:tools}, we develop several tools in Malliavin calculus that play a central role in our approach. Section \ref{sec:smoothness} is then devoted to the study of the Malliavin smoothness of the solution to \eqref{eq:SDEsDrivenF}. As a consequence, we show in Section \ref{sec:density} that this solution admits a density with respect to the Lebesgue measure. Finally, Section \ref{sec:Hermite} applies our approach to stochastic differential equations driven by a multidimensional Hermite process.

Throughout the paper, we use the symbol $\bullet$ for stochastic arguments, so that $X(\bullet)$ denotes the random variable $\omega\mapsto X(\omega)$, and the symbol $\star$ for deterministic arguments, so that $f(\star)$ stands for the map $t\mapsto f(t)$.

\section{Background and strategies} \label{sec:backgroundstrat}

\subsection{Abstract Wiener spaces} \label{subsec:abstractwienerspaces}

An abstract Wiener space is a quadruple $(\Omega,\mathcal{F},\P,\mathfrak{H})$ where
\begin{enumerate}[(a)]
\item $\Omega$ is a separable Banach space;
\item $\P$ is a Gaussian measure on $\Omega$ with full support;
\item $\mathcal{F}$ is the completion with respect to $\P$ of the Borel $\sigma$-algebra on $\Omega$;
\item $(\mathfrak{H}, \langle \cdot , \cdot \rangle_{\mathfrak{H}})$ is a separable Hilbert space that is continuously and densely embedded in $\Omega$ via an injection $j \, : \, \mathfrak{H} \to \Omega$.
\end{enumerate}
The image $j(\mathfrak{H})$ is usually called the Cameron--Martin space. By the Riesz representation theorem, one identifies $\mathfrak{H}$ with its dual and considers the adjoint operator $j^* \, :\, \Omega^* \to \mathfrak{H}$ of $j$, which is characterised by
\begin{equation}\label{eqn:operateurdual}
\langle j^*(y),h \rangle_\mathfrak{H}=y(j(h))
\end{equation}
for $y \in \Omega^*$ and $h \in \mathfrak{H}$. We further assume that, for every $y \in \Omega^*$,
\begin{equation}\label{eqn:caractofy}
\int_{\Omega} e^{iy(\omega)} d\mathbb{P}(\omega)= \exp \left( - \frac{1}{2}\| j^*(y) \|_{\mathfrak{H}}^2 \right).
\end{equation}
The subspace $j^*(\Omega^*)$ is dense in $\mathfrak{H}$, and this allows one to define, for every $g \in \mathfrak{H}$, a centered Gaussian random variable $X_g$. Whenever $g= j^*(y)$ for some $y \in \Omega^*$, it suffices, in view of \eqref{eqn:caractofy}, to set $X_g=y$. The same identity shows that the map $g \mapsto X_g$ from $j^*(\Omega^*)$ to $L^2(\Omega)$ is continuous (with respect to the norm of $\mathfrak{H}$), and therefore extends continuously to the whole space $\mathfrak{H}$. The resulting Gaussian process $\{X_g \, : \, g \in \mathfrak{H}\}$ on $(\Omega,\mathcal{F},\P)$ is \emph{isonormal}, in the sense that
\[  \E [X_gX_{g'}] = \langle g,g' \rangle_{\mathfrak{H}} \qquad \text{for all } g,g' \in \mathfrak{H}.\]
Such a process is the natural framework in which to set up Malliavin calculus. Below, we collect the main ingredients needed in the present paper and refer to \cite{nourdin2012normal,Nualart} for details and proofs.

\paragraph{Malliavin derivatives and Sobolev-type spaces.}
Throughout this paper, we adopt the usual notation $\otimes$ for the tensor product of elements of $\mathfrak{H}$, and $\odot$ for the symmetric tensor product, that is, given $n \in \N$ and $h_1,\dots,h_n \in \mathfrak{H}$, 
\[ h_1 \odot \dots \odot h_n := \frac{1}{n!} \sum_{\sigma \in \mathfrak{S}_n} h_{\sigma(1)} \otimes \dots \otimes h_{\sigma(n)},\]
with $\mathfrak{S}_n$ the set of all permutations of $\{1,\dots,n\}$.

For each $q \geq 1$, we write $\mathfrak{H}^{\otimes  q}$ for the $q$th tensor product of $\mathfrak{H}$, $\mathfrak{H}^{\odot q}$ for its $q$th \textit{symmetric} tensor product, and $L^2(\Omega,\mathfrak{H}^{\odot q})$ for the space of $\mathfrak{H}^{\odot q}$-valued random elements $F$ that are $\mathcal{F}$-measurable and satisfy $\mathbb{E} [\| F\|_{\mathfrak{H}^{\odot q}}^2 ] < \infty$. We denote by $\mathcal{S}$ the class of cylindrical random variables of the form
\begin{equation}\label{intro:defmallia}
F=f(X_{h_1},\ldots,X_{h_n}),
\end{equation}
where $n \geq 1$, $h_1, \dots, h_n \in \mathfrak{H}$, and $f$ is infinitely differentiable with all partial derivatives of polynomial growth. For such an $F$, the $q$th Malliavin derivative is the element of $L^2(\Omega,\mathfrak{H}^{\odot  q})$ given by
\begin{equation}\label{eqn:defmalderi}
D^q F = \sum_{\ell_1,\ldots,\ell_q=1}^n \frac{\partial^q f}{\partial \ell_{1} \ldots \partial \ell_{k}}(X_{h_1}, \ldots, X_{h_n}) h_{\ell_1} \otimes \cdots \otimes h_{\ell_q}.
\end{equation}
The operator $D^q$ is closable (see e.g.\ \cite[Proposition 2.3.4]{nourdin2012normal}); for every $m \in \N^*$ and $p \geq 1$, the Sobolev-type space $\mathbb{D}^{m,p}$ is then defined as the closure of $\mathcal{S}$ with respect to the norm
\begin{equation}\label{eqn:defofthenorms}
\| \cdot \|_{m,p} \, : \, F \mapsto \left(\mathbb{E}[|F|^p] + \sum_{q=1}^m \mathbb{E}[ \|D^q F \|^p_{\mathfrak{H}^{\otimes k}}] \right)^\frac{1}{p}.
\end{equation}
Finally, we set $\mathbb{D}^{\infty}=\bigcap_{p\geq 1}\bigcap_{m\geq 1}\mathbb{D}^{m,p}$.

\paragraph{Divergence operator and multiple integrals.}
The divergence operator of order $q$, denoted by $\delta^q$, is defined as the adjoint of $D^q$. It is an unbounded operator from $L^2(\Omega,\HHH^{\otimes q})$ into $L^2(\Omega)$ whose domain $\Dom(\delta^q)$ consists of those $u\in L^2(\Omega,\HHH^{\otimes q})$ for which there exists a constant $c>0$ such that $|\E[\left<D^qF,u\right>_{\HHH^{\otimes q}}]|\leq c\|F\|_{L^2(\Omega)}$ for every $F\in \mathbb{D}^{q,2}$. For $u\in \Dom(\delta^q)$, $\delta^q(u)$ is the unique element of $L^2(\Omega)$ characterised by the duality relation
\begin{equation}\label{duality}
  \E[F\delta^q(u)]=\E[\left<D^qF,u\right>_{\HHH^{\otimes k}}],\quad \text{for all} \;F\in\mathbb{D}^{k,2}.
\end{equation}
A direct consequence of this definition is the ``integration by parts'' formula \cite[Proposition 2.5.4]{nourdin2012normal}: if $F \in \mathbb{D}^{1,2}$ and $u \in \Dom(\delta)$ are such that $\E[F^2\|u\|^2_\mathfrak{H}]$, $\E[F^2\delta(u)^2]$ and $\E[\langle Df,u\rangle_\mathfrak{H}^2]$ are all finite, then $Fu \in \Dom(\delta)$ and
\begin{equation}\label{integrationbypart}
\delta(Fu)=F\delta(u)-\langle DF,u\rangle_\mathfrak{H}.
\end{equation}
The divergence operator is the key ingredient in the construction of multiple stochastic integrals, which play a central role in our work.

Let $q \geq 1$ and $f\in \HHH^{\odot  q}$. The multiple Wiener-It\^o integral of order $q$ of $f$ is the centered random variable
\begin{equation}\label{def:stochasticintegrals}
I_q(f):=\delta^q(f).
\end{equation}

Among the many properties of multiple Wiener-It\^o integrals, the following ones will be used repeatedly.
\begin{itemize}
	\item \textit{Hypercontractivity.} For every $q \geq 1$,  $p \in[2, \infty)$ and $f\in \HHH^{\odot  q}$,
\begin{equation}\label{eq:hyper}
\E\left[\left|I_q(f)\right|^p\right]^{1 / p} \leq (p-1)^{q/2}\, \E\left[\left|I_q(f)\right|^2\right]^{1 / 2} .
\end{equation}
	\item \textit{Isometry formula} (when $p=q$) and \textit{orthogonality} (when $p \neq q$): for $f \in \mathfrak{H}^{\odot  p}$, $g \in \mathfrak{H}^{\odot  q}$ and $p,q \geq 1$,
\begin{equation}\label{eq:isom}
\mathbb{E}\left[I_p(f) I_q(g)\right]= \begin{cases}p !\langle f, g\rangle_{\mathfrak{H}^{\otimes p}} & \text { if } p=q, \\ 0 & \text { if } p \neq q .\end{cases}
\end{equation}
\item \textit{Product formula for multiple Wiener--It\^o integrals:} for any $f\in\mathfrak H^{\odot p}$ and $g\in\mathfrak H^{\odot q}$,
\begin{equation}\label{eq:product-formula}
I_p(f)\,I_q(g)
=
\sum_{r=0}^{p\wedge q}
r!\binom{p}{r}\binom{q}{r}\,
I_{p+q-2r}\big(f\widetilde\otimes_r g\big),
\end{equation}
where $f\widetilde\otimes_r g$ denotes the symmetrised $r$th contraction of $f$ and $g$.
\end{itemize}

\paragraph{Hilbert-space-valued random variables.}
We shall also need random variables taking values in a separable Hilbert space, and we thus extend the above definitions accordingly. Let $\mathcal{U}$ be a separable Hilbert space. We denote by $\mathcal{S}_{\mathcal{U}}$ the class of $\mathcal{U}$-valued smooth random variables of the form $F=\sum_{j=1}^{n}F_j\,u_j$, with $F_j\in \mathcal{S}$ and $u_j\in \mathcal{U}$. For any such $F$ and any $q \geq 1$, the $q$th Malliavin derivative of $F$ is the $\mathfrak{H}^{\otimes q} \otimes \mathcal{U}$-valued random element
\[ D^q F = \sum_{j=1}^n D^q F_j \otimes u_j.\]
As before, one shows that, for every $p,q \geq 1$, $D^q$ is closable from $\mathcal{S}_{\mathcal{U}} \subset L^p(\Omega,\mathcal{U})$ into $L^p(\Omega,\mathfrak{H}^{\otimes q} \otimes \mathcal{U})$. This allows us to introduce the analogous Sobolev-type spaces $\mathbb{D}^{m,p}(\mathcal{U})$ and $\mathbb{D}^{\infty}(\mathcal{U})$, together with the norm
\begin{equation}\label{eqn:defofthenormsU}
\| \cdot \|_{m,p,,\mathcal{U}} \, : \, u \mapsto \left(\mathbb{E}[\|u\|_{\mathcal{U}}^p] + \sum_{q=1}^m \mathbb{E}[ \|D^q u \|^p_{\mathfrak{H}^{\otimes k}\otimes\mathcal{U}}] \right)^\frac{1}{p}.
\end{equation}
Similarly, for $q \geq 1$, if $u \in \mathcal{U} \otimes \mathfrak{H}^{\otimes q}$ has the form $u=\sum_{j=1}^n u_j \times h_j$ with $u_j \in \mathcal{H}$ and $h_j \in \mathfrak{H}^{\otimes q}$, we set
\[ \delta^q(u)=\sum_{j=1}^n u_j \delta^q(h_j),\]
and $\delta^q$ extends to a bounded operator from $\mathcal{U} \otimes \mathfrak{H}^{\otimes q}$ to $L^2(\Omega,\mathcal{U})$ \cite[Section 2.6]{nourdin2012normal}. In particular, for any $p>q$, the identification $\mathfrak{H}^{\otimes p}=\mathfrak{H}^{\otimes {(p-q)}} \otimes \mathfrak{H}^{\otimes q}$ gives meaning to $\delta^q(f)$ for $f \in \mathfrak{H}^{\otimes p}$. Extending definition (\ref{def:stochasticintegrals}) in this way, we obtain two further properties of multiple Wiener-It\^o integrals that will be used in the sequel:
\begin{itemize}
\item \textit{Malliavin derivative of a multiple Wiener-It\^o integral.} For every $p,q \geq 1$ and $f\in \HHH^{\odot  q}$, one has $I_q(f) \in \mathbb{D}^{\infty,p}$, and, for every $r \geq 1$,
\begin{equation}\label{eqn:derivativeofintegral}
D^r I_q(f)= \begin{cases}\frac{q!}{(q-r)!} I_{q-r}(f) & \text { if } r \leq q,\\ 0 & \text { otherwise,}\end{cases}
\end{equation}
where $I_{q-r}(f)$ is computed with respect to the first $q-r$ variables.
\item  \textit{Reintegration formula:} for every $q \geq 1$ and $f\in \HHH^{\odot  q}$,
\begin{equation}\label{eqn:reintegration}
I_q(f)=\delta(I_{q-1}(f)).
\end{equation}
\end{itemize}

\subsection{Existence of the solution and statement of the main conditions}\label{subsec:existencesol}

From now on, we fix an abstract Wiener space $(\Omega,\mathcal{F},\P,\mathfrak{H})$ and assume that the driving process $F$ of \eqref{eq:SDEsDrivenF} takes the following form.

\begin{Def}\label{def:drivingprocess}
An $\R^m$-valued stochastic process $\{F_t\}_{t \geq 0}$ is called a \emph{chaotic process} (over the abstract Wiener space $(\Omega,\mathcal{F},\P,\mathfrak{H})$) if there exists $q \in \N$ such that, for every $t \geq 0$ and every $1 \leq \ell \leq m$, one can find $f_t^\ell \in \mathfrak{H}^{\odot  q}$ satisfying
\begin{equation}\label{eqn:intfordrivingprocess}
F_t=(I_q(f_t^1),\dots,I_q(f_t^m)).
\end{equation}
\end{Def}

The aim of the present subsection is to discuss the existence and uniqueness of a solution to \eqref{eq:SDEsDrivenF}. For $p \in \N$, we write $\mathcal{C}^p_b(\R^d)$ for the set of real-valued functions on $\R^d$ that are $p$-times continuously differentiable with bounded partial derivatives up to order $p$. We now impose the following assumption on the coefficients $(b_k)_{1 \leq k \leq d}$ and $(\sigma_{k,\ell})_{1 \leq k \leq d, 1 \leq \ell \leq m}$ of \eqref{eq:SDEsDrivenF}.

\begin{center}
$\mathbf{(H_{1})}$ For all $1 \leq k \leq d$, and $1 \leq \ell \leq m$, $b_k, \sigma_{k,\ell}\in \mathcal{C}^3_b(\R^d)$.
\end{center}

In \cite{MR1893308}, the authors consider the deterministic differential equation on $\mathbb{R}^{d}$ driven by an $\R^m$-valued function $\varphi$,
\begin{equation}\label{eq:DeterministicSDE}
x_{t}^{k}=x_{0}^{k}+\int_{0}^{t} b_{k}\left(x_{s}\right) \mathrm{d} s+\sum_{\ell=1}^{m} \int_{0}^{t} \sigma_{k,\ell}\left(x_{s}\right) \mathrm{d} \varphi_{s}^{\ell},
\end{equation}
with $1 \leq k \leq d$, $t \in[0, T]$ and $x_0\in\R^d$, and they establish its existence and uniqueness in terms of the function spaces recalled below.

\begin{Def}\label{def:functionspaces}
Let $T>0$. For $\alpha \in (0,\tfrac{1}{2})$ and $\theta \in (0,1)$, we introduce the following spaces of measurable functions:
\begin{enumerate}[(a)]
\item the space $W_{1}^{\alpha}(0, T; \mathbb{R}^{d})$ of functions $f:[0, T] \rightarrow \mathbb{R}^{d}$ satisfying
$$
\|f\|_{\alpha, 1}:=\sup _{t \in[0, T]} \left( |f(t)|+\int_{0}^{t} \frac{|f(t)-f(s)|}{|t-s|^{\alpha+1}} \mathrm{~d} s\right)<\infty;
$$
\item the space $W_{2}^{1-\alpha}(0, T; \mathbb{R}^{m})$ of functions $\varphi:[0, T] \rightarrow \mathbb{R}^{m}$ satisfying
$$
\|\varphi\|_{1-\alpha, 2}:=\sup _{0 \leq s<t \leq T}\left(\frac{|\varphi(t)-\varphi(s)|}{(t-s)^{1-\alpha}}+\int_{s}^{t} \frac{|\varphi(\tau)-\varphi(s)|}{(\tau-s)^{2-\alpha}} \mathrm{d} \tau\right)<\infty;
$$
\item the $\theta$-H\"older space $C^{\theta}(0, T; \mathbb{R}^{d})$ of functions $f:[0, T] \rightarrow \mathbb{R}^{d}$ satisfying
$$
\|f\|_{\theta}:= \sup_{t \in [0,T]} |f(t)|+\sup _{0 \leq s<t \leq T} \frac{|f(t)-f(s)|}{(t-s)^{\theta}} < \infty.
$$
\end{enumerate}
\end{Def}

\begin{Rmk}
For any $T>0$, $\alpha \in (0,\frac{1}{2})$ and $\gamma >0$ such that $1-\alpha + \gamma \leq 1$, one has the continuous embeddings
\begin{equation}\label{eq:injection1}
C^{\alpha+\gamma}\left(0, T; \mathbb{R}^{d}\right) \subset W_{1}^{\alpha}\left(0, T; \mathbb{R}^{d}\right)
\end{equation}
and
\begin{equation}\label{eq:injection2}
C^{1-\alpha+\gamma}\left(0, T; \mathbb{R}^{m}\right) \subset W_{2}^{1-\alpha}\left(0, T; \mathbb{R}^{m}\right) \subset C^{1-\alpha}\left(0, T; \mathbb{R}^{m}\right).
\end{equation}
\end{Rmk}

Our starting point is \cite[Theorem 5.1]{MR1893308}, which asserts that, for any $\alpha \in (0,\frac{1}{2})$ and under assumptions on the coefficients weaker than $\mathbf{(H_1)}$, if $\varphi \in W_{2}^{1-\alpha}\left(0, T; \mathbb{R}^{m}\right)$, then the deterministic differential equation \eqref{eq:DeterministicSDE} admits a unique solution $x \in W_{1}^{\alpha}\left(0, T; \mathbb{R}^{d}\right) \cap C^{1-\alpha}\left(0, T; \mathbb{R}^{d}\right)$. To apply this result to the stochastic differential equation \eqref{eq:SDEsDrivenF}, we need to ensure that (a version of) the driving process $F$ has trajectories in some space $W_{2}^{1-\alpha}\left(0, T; \mathbb{R}^{m}\right)$. This leads us to impose the following assumption on the chaotic process $F$ given by \eqref{eqn:intfordrivingprocess}.

\begin{center}
$\mathbf{(H_{2})}$ There exists $H \in (\frac{1}{2},1)$ such that, for any $1 \leq \ell \leq m$, there exists a deterministic constant $c_\ell>0$ for which, for any $s,t \in [0,T]$,
\begin{equation}\label{eq:HolderF}
	\|f^\ell_{t}-f^\ell_{s}\|_{\HHH^{\otimes q}}\leq c_\ell\,|t-s|^{H}.
\end{equation}
\end{center}
Whenever $\mathbf{(H_{2})}$ holds, the isometry formula \eqref{eq:isom} together with the hypercontractivity property \eqref{eq:hyper} yield, for every $1 \leq \ell \leq m$, $0 \leq j \leq q$, $p \geq 2$ and $s,t \in [0,T]$,
\begin{equation}\label{eq:HolderF0}
	\E\left[ \left\|D^j(F_t^\ell-F_s^\ell)\right\|_{\mathfrak{H}^{\otimes j}}^p\right]\leq c_{\ell,j,p}\,|t-s|^{pH}.
\end{equation}
From \eqref{eq:HolderF0} and Kolmogorov's continuity theorem, we deduce that, for every $0 \leq j \leq q$, the process $\{D^j F_t\}_{t \in [0,T]}$ admits a version whose trajectories are $\theta$-H\"{o}lder continuous as maps from $[0,T]$ to $\mathfrak{H}^{\otimes j}$\, for every $\theta\in (0, H)$. In the sequel, whenever $\mathbf{(H_{2})}$ is in force, we shall always work with these versions, still denoted $\{D^j F_t\}_{t \in [0,T]}$. In particular, for any $\beta \in (0,H-\frac{1}{2})$ and $\gamma \in (0,\beta)$, we have $H-\beta+\gamma \leq 1$, and the trajectories of $\{F_t\}_{t \in [0,T]}$ therefore belong to $C^{H-\beta + \gamma}\left(0, T; \mathbb{R}^{m}\right)\subseteq W_{2}^{H-\beta}\left(0, T; \mathbb{R}^{m}\right)$, so that \cite[Theorem 5.1]{MR1893308} can be applied with $1-\alpha=H-\beta$.

We close this subsection by recalling \cite[Proposition 4]{nualart2009malliavin}, which plays a key role throughout our analysis.

\begin{Prop}\label{pro:Nualart}
Fix $\alpha \in (0,\frac{1}{2})$ and $T>0$, and let $x(\varphi)$ denote the solution to \eqref{eq:DeterministicSDE} driven by $\varphi \in W_{2}^{1-\alpha}\left(0, T; \mathbb{R}^{m}\right)$. Under hypothesis $\mathbf{(H_{1})}$, the mapping
\begin{equation}\label{eqn:frechetdifferentiabilityPSI}
\Psi \, : \, W_{2}^{1-\alpha}\left(0, T; \mathbb{R}^{m}\right) \to W_{1}^{\alpha}\left(0, T; \mathbb{R}^{d}\right) \, : \, \varphi \mapsto x(\varphi)
\end{equation}
is Fréchet differentiable and for any $\varphi, \psi \in W_{2}^{1-\alpha}\left(0, T ; \mathbb{R}^{m}\right)$ the Fr\'echet derivative of $\Psi$ at $\varphi$ in the direction $\psi$ is given by
$$
\mathcal{D} \Psi (\varphi) [\psi] \, : \, t \mapsto \left( \sum_{\ell=1}^{m} \int_{0}^{t} \Theta_{t}^{k,\ell}(s) \mathrm{d} \psi_{s}^{l}\right)_{1 \leq k \leq d},
$$
where for $1 \leq k \leq d$, $1 \leq \ell \leq m$,  $0 \leq s \leq t \leq T$, $\Theta_{t}^{k,\ell}(s)$ satisfies
$$
\Theta_{t}^{k,\ell}(s)=\sigma_{k,\ell}\left(x_{s}\right)+\sum_{p=1}^{d} \int_{s}^{t} \partial_{p} b_{k}\left(x_{u}\right) \Theta_{u}^{p, \ell}(s) \mathrm{d} u+\sum_{p=1}^{d} \sum_{q=1}^{m} \int_{s}^{t} \partial_{p} \sigma^{k,q}\left(x_{u}\right) \Theta_{u}^{p, \ell}(s) d\varphi_{u}^{q},
$$
and $\Theta_{t}^{k ,\ell}(s)=0$ if $s>t$.
\end{Prop}

We also state the following facts, gathered from \cite[Lemma 3]{nualart2009malliavin} and the proof of \cite[Proposition 4]{nualart2009malliavin}. They will play a crucial role in our proof of the absolute continuity of the law of the solution to \eqref{eq:SDEsDrivenF} in Section \ref{sec:density}.

\begin{Prop}\label{pro:Nualartbis}
Under the conditions and with the notations of Proposition \ref{pro:Nualart}, the mapping
\begin{align}\label{eqn:Nualartbis}
T & \, : \, W_{2}^{1-\alpha}\left(0, T; \mathbb{R}^{m}\right) \times  W_{1}^{\alpha}\left(0, T; \mathbb{R}^{d}\right) \to  W_{1}^{\alpha}\left(0, T; \mathbb{R}^{d}\right) \nonumber \\ &\, : \, (\psi,x) \mapsto x-x_0-\int_0^\star b(x_s) ds - \int_0^\star \sigma(x_s) d(\varphi_s+\psi_s)
\end{align}
is Fréchet differentiable with, for any $(\psi,x) \in W_{2}^{1-\alpha}\left(0, T; \mathbb{R}^{m}\right) \times  W_{1}^{\alpha}\left(0, T; \mathbb{R}^{d}\right)$, $\phi \in W_{2}^{1-\alpha}\left(0, T; \mathbb{R}^{m}\right) $ and $y \in W_{1}^{\alpha}\left(0, T; \mathbb{R}^{d}\right)$, for all $1 \leq k \leq d$,
\begin{equation}\label{eqn:Nualartbis:frechet1}
(\mathcal{D}_1 T (\psi,x) [\phi])^k = - \sum_{\ell=1}^m \int_0^\star \sigma_{k,\ell}(x_s) d\phi_s^\ell
\end{equation}
and
\begin{equation}\label{eqn:Nualartbis:frechet2}
(\mathcal{D}_2 T (\psi,x) [y])^k=y^k-\sum_{p=1}^d \int_0^\star \partial_p b_k(x_s)y_s^k ds - \sum_{p=1}^d \sum_{\ell=1}^m \int_0^\star \partial_k \sigma_{p,\ell} d(\varphi_s^\ell+\psi_s^\ell).
\end{equation}
Furthermore, $\mathcal{D}_2 T (0,x)$ is a linear homeomorphism from $W_{1}^{\alpha}\left(0, T; \mathbb{R}^{d}\right)$ to\\  $C^{1-\alpha}\left(0, T; \mathbb{R}^{d}\right)$ and we have, for any $\varphi \in W_{2}^{1-\alpha}\left(0, T ; \mathbb{R}^{m}\right)$
\begin{equation}\label{eqn:Nualartbis:implict}
\mathcal{D} \Psi (\varphi)=-\mathcal{D}_2 T (0,x)^{-1} \circ \mathcal{D}_1 T (0,x).
\end{equation}
\end{Prop}

\subsection{Roadmap through an alternative definition of Malliavin calculus}\label{sub:malialter}

In Subsection \ref{subsec:abstractwienerspaces}, we sketched Shigekawa's definition (originally introduced by Malliavin \cite{MR0536013} and Shigekawa \cite{MR0582167}) of the Sobolev-type spaces $\mathbb{D}^{m,p}(\mathcal{U})$, where $\mathcal{U}$ is a separable Hilbert space. While this is the most commonly used definition in the literature, these spaces admit several equivalent characterisations. The one we shall exploit here is due to Stroock \cite{MR0603973} and Kusuoka \cite{MR0657873} (see also Sugita \cite{MR0810975}), who introduced the Sobolev spaces $\hat{\mathbb{D}}^{m,p}(\mathcal{U})$ (see the definition below). Their definition is based on the following two properties.

Denoting by $\mathcal{L}_{HS}(j(\mathfrak{H}),\mathcal{U})$ the space of Hilbert--Schmidt operators from $j(\mathfrak{H})$ to $\mathcal{U}$, these two properties read as follows.

\begin{Def}\label{def:RACSGD}
A $\mathcal{U}$-valued random variable $F : \Omega \to \mathcal{U}$ is said to be
\begin{enumerate}[(a)]
\item \emph{ray absolutely continuous} \textbf{(RAC)} if, for every $h\in \mathfrak{H}$, there exists a $\mathcal{U}$-valued random variable $\hat{F}_h : \Omega \to \mathcal{U}$ such that $F=\hat{F}_h$ $\mathbb{P}$-a.s.\ and, for every $\omega \in \Omega$, the map $\mathbb{R} \ni\varepsilon\mapsto \hat{F}_h(\omega+\varepsilon\,j(h))$ is absolutely continuous;
\item \emph{stochastically Gâteaux differentiable} \textbf{(SGD)} if there exists a random variable $G: \Omega\to \mathcal{L}_{HS}(j(\mathfrak{H}),\mathcal{U})$ such that, for every $h\in \mathfrak{H}$,
\begin{equation}\label{eq:SGD}
\frac{1}{\varepsilon}(F(\bullet + \varepsilon\,j(h))-F) \overset{\mathbb{P}}{\longrightarrow} G(\bullet)[j(h)] \qquad\text{ as } \varepsilon\to 0.
\end{equation}
The operator $G$ is unique $\mathbb{P}$-a.s., and is denoted by $\hat{D}F$. Higher-order derivatives are defined inductively: if $\hat{D}^{n-1}F$ is \textbf{(SGD)}, then $\hat{D}^{n}F:=\hat{D}(\hat{D}^{n-1}F)$.
\end{enumerate}
\end{Def}

The Sobolev-type spaces $\hat{\mathbb{D}}^{m,p}(\mathcal{U})$, for $m\in \N^*$ and $1<p<\infty$, are then defined inductively. First, for $m=1$,
\begin{equation}\label{eq:TildeD1p}
	\hat{\mathbb{D}}^{1,p}(\mathcal{U})=\{F\in L^p(\mathcal{U})\,;\; F \text{ is RAC and SGD, } \hat{D}F \in L^p\left(\mathcal{L}_{HS}(j(\mathfrak{H}),\mathcal{U})\right)\},
\end{equation}
and, for $m >1$, one then sets analogously
\begin{equation}\label{eq:TildeDmp}
	\hat{\mathbb{D}}^{m,p}(\mathcal{U})=\{F\in \hat{\mathbb{D}}^{m-1,p}(\mathcal{U})\,;\;  \hat{D}F \in \hat{\mathbb{D}}^{m-1,p}\left(\mathcal{L}_{HS}(j(\mathfrak{H}),\mathcal{U}))\right)\}.
\end{equation}

The following theorem shows that Kusuoka-Stroock approach and Shigekawa's definition lead to the same spaces.
\begin{Thm}[Theorem 3.1 in \cite{MR0810975}]\label{thm:KusuokaStroock=Shigekawa}
	For $m\in \N^*$ and $1<p<\infty$, we have $\hat{\mathbb{D}}^{m,p}(\mathcal{U})=\mathbb{D}^{m,p}(\mathcal{U})$, and for any $F$ in this space, we have, for a.e. $\omega\in\Omega$, and for any $h\in \HHH$,
	\begin{equation}\label{eq:TildeDEqualD}
		\hat{D}^{k}F(\omega)[j(h),\ldots,j(h)]=\left<D^{k}F(\omega),h^{\otimes k}\right>_{\HHH^{\otimes k}}, \qquad \text{ for } 1 \leq k \leq m. 
	\end{equation}
\end{Thm}

In the present work, we shall rely on the Kusuoka-Stroock approach to establish the Malliavin differentiability of the solution $(X_t)_{t \in [0,T]}$ to the stochastic differential equation \eqref{eq:SDEsDrivenF}. Let us now briefly outline our strategy.

\paragraph{Strategy for \textbf{(SGD)}.}
For any $t \in [0,T]$, $1 \leq k \leq d$ and $h \in \mathfrak{H}$, we have to analyse the limit in probability, as $\varepsilon \to 0$, of
\begin{equation}\label{eqn:limiteacalculerpourSGD}
\frac{X_t^k(\bullet+\varepsilon j(h))-X_t^k(\bullet)}{\varepsilon}.
\end{equation}
Recall that, by Proposition \ref{pro:Nualart}, for $\varphi \in W_{2}^{1-\alpha}\left(0, T; \mathbb{R}^{m}\right)$, $\Psi(\varphi)$ denotes the solution to the deterministic equation \eqref{eq:DeterministicSDE} driven by $\varphi$ on $[0,T]$, and $\Psi_t(\varphi)$ is its value at time $t$. Under hypothesis $\mathbf{(H_2)}$, we can therefore write, for any $\omega \in \Omega$ and $\varepsilon \in \R$,
\[ X_t^k(\omega+\varepsilon j(h))= \Psi_t^k(F_\star(\omega+\varepsilon j(h)).\]
Since $\Psi$ is Fréchet differentiable by Proposition \ref{pro:Nualart}, a naive approach would consist in applying the chain rule and studying the differentiability of the map
\[ \Theta \, : \, \R \to W_2^{1-\alpha}(0,T;\R^m) \, : \, \varepsilon \mapsto F_\star(\omega + \varepsilon j(h)).\]
Unfortunately, this strategy is hopeless in our general setting. A good illustration is given by the standard case of the canonical Brownian motion on
\[ C_0([0,T]) := \{\omega \,: \, [0,T] \to \R \text{ continuous with } \omega_0=0\},\]
equipped with its Borel $\sigma$-algebra and the Wiener measure. It is shown in \cite[Subsection 1.B]{Nualart1990} that, for $f \in L^2([0,T))$, the random variable $\omega \mapsto I_1(f)(\omega)$ admits a continuous extension for the supremum norm if and only if $f$ has bounded variation. Since even the continuity of the map $\omega \mapsto I_1(f)(\omega)$ fails in such a basic situation, expecting $\Theta$ to be differentiable is far too optimistic.

However, since the limit in \eqref{eq:SGD} is only required to hold in probability, it is enough to replace \eqref{eqn:limiteacalculerpourSGD} with the limit in probability of
\begin{equation}\label{eqn:limiteacalculerpourSGD2}
\frac{\Gamma_{h,t}^k(\varepsilon,\bullet)-\Gamma_{h,t}^k(0,\bullet)}{\varepsilon},
\end{equation}
where:
\begin{itemize}
\item for each fixed $t \in [0,T]$, $h \in \mathfrak{H}$ and $\varepsilon \in \R$, we have, for almost every $\omega \in \Omega$,
\[ \Gamma_{h,t}(\varepsilon,\omega)=X_t(\omega + \varepsilon j(h)),\]
i.e.\ $\{\Gamma_{h,t}(\varepsilon,\bullet) \, : \, t \in [0,T], \, \varepsilon \in \R, \, h \in \mathfrak{H} \}$ is a version of $\{X_t(\bullet+ \varepsilon j(h)) \, : \, t \in [0,T], \, \varepsilon \in \R, \, h \in \mathfrak{H} \}$;
\item the limit \eqref{eqn:limiteacalculerpourSGD2} is more tractable than \eqref{eqn:limiteacalculerpourSGD}.
\end{itemize}
Since the map $\Theta$ was clearly the problematic ingredient above, we shall construct an $\R^m$-valued stochastic process $\{S_{t,\varepsilon}^h \, : \, t \in [0,T], \, \varepsilon \in \R, \, h \in \mathfrak{H} \}$ enjoying the following properties:
\begin{itemize}
\item for each $h \in \mathfrak{H}$ and $\varepsilon \in \R$, the map $t \mapsto S_{t,\varepsilon}^h$ is continuous;
\item for each $h \in \mathfrak{H}$ and $t \in [0,T]$, the map $\varepsilon \mapsto S_{t,\varepsilon}^h$ is continuously differentiable;
\item for every $h \in \mathfrak{H}$, $\varepsilon \in \R$ and $t \in [0,T]$, and for almost every $\omega \in \Omega$,
\[S_{t,\varepsilon}^h(\omega)=F_t(\omega + \varepsilon j(h)),\]
so that $\{S_{t,\varepsilon}^h \, : \, t \in [0,T], \, \varepsilon \in \R, \, h \in \mathfrak{H} \}$ is a version of $\{F_t(\bullet+ \varepsilon j(h)) \, : \, t \in [0,T], \, \varepsilon \in \R, \, h \in \mathfrak{H} \}$.
\end{itemize}

The first two subsections of Section \ref{sec:tools} are devoted to this task:
\begin{enumerate}
\item in Subsection \ref{sub:taylor}, we establish a Taylor formula for Malliavin calculus, which, in our setting, applies to the driving process $F$;
\item in Subsection \ref{sub:malcalsdes}, we combine this Taylor formula with Proposition \ref{pro:Nualart} to construct the process $\{S_{t,\varepsilon}^h \, : \, t \in [0,T], \, \varepsilon \in \R, \, h \in \mathfrak{H} \}$ explicitly;
\item the process $\{\Gamma_{h,t}(\varepsilon,\bullet) \, : \, t \in [0,T], \, \varepsilon \in \R, \, h \in \mathfrak{H} \}$ is then defined, for $\varepsilon \in \R$ and $h \in \mathfrak{H}$, by
\[\Gamma_{h,\star}(\varepsilon,\bullet)=\Psi(S_{\star,\varepsilon}^h).\]
\end{enumerate}

\paragraph{Strategy for \textbf{(RAC)}.}
For each $t \in [0,T]$ and $h \in \mathfrak{H}$, we have to construct a random variable $\widetilde{X}^h_t$ satisfying:
\begin{itemize}
\item  $\widetilde{X}^h_t=X_t$ almost surely;
\item for every $\omega \in \Omega$, the map $\varepsilon \mapsto \widetilde{X}^h_t(\omega+\varepsilon j(h))$ is absolutely continuous.
\end{itemize}
Subsection \ref{sub:version} is devoted to the preparation of this task. We proceed in two steps:
\begin{enumerate}
\item in Subsection \ref{sub:taylor}, we construct a version of the isonormal Gaussian process $\{X_g \, : \, g \in \mathfrak{H}\}$ on $(\Omega,\mathcal{F},\mathbb{P},\mathfrak{H})$ that enjoys specific key properties with respect to translations in the direction of $j(h)$;
\item this version is then used to build, for each $h \in \mathfrak{H}$, a version $\widetilde{F}^h$ of the driving process such that, for every $\omega \in \Omega$, the map $(t,\varepsilon) \mapsto \widetilde{F}^h_t(\omega + \varepsilon j(h))$ is continuous and the map $\varepsilon \mapsto \widetilde{F}^h_\star(\omega + \varepsilon j(h))$ is smooth.
\end{enumerate}

\section{Tools in Malliavin calculus} \label{sec:tools}

\subsection{A Taylor formula for Malliavin Calculus}\label{sub:taylor}

Both the \textbf{(RAC)} and \textbf{(SGD)} conditions require to handle, for $t \geq 0$, $\omega \in \Omega$ and $h \in \mathfrak{H}$, the quantity $X_t(\omega+j(h))$ (and hence $F_t(\omega+j(h))$). Our goal in this section is to derive a Taylor expansion for this last expression. As a first step, we build a version of the isonormal process $\{X_g \, : \, g \in \mathfrak{H}\}$ (introduced in Subsection \ref{subsec:abstractwienerspaces}) for which such an expansion is available.

\begin{Prop}\label{prop:casdebase}
For every $g\in \mathfrak{H}$, there exist two events $\Omega_g^\prime \subseteq \Omega_g$ with $\mathbb{P}(\Omega_g^\prime)=\mathbb{P}(\Omega_g)=1$ and a random variable $\widetilde{X}_g$ such that
\begin{enumerate}[(a)]
\item  $\Omega_g + j(\mathfrak{H}) \subseteq \Omega_g$;
\item the equality $X_g=\widetilde{X}_g$ holds on $\Omega_g^\prime$;
\item  for all $\omega \in \Omega_g$ and $h \in \mathfrak{H}$,
\begin{equation}\label{eqn:fundamental}
\widetilde{X}_g(\omega + j(h))=\widetilde{X}_g(\omega)+ \langle g , h \rangle_\mathfrak{H}.
\end{equation}
\end{enumerate}
\end{Prop}

\begin{proof}
We begin with the case where $g=j^*(y)$ for some $y \in \Omega^*$. Then, using the definition of $X_g$ together with \eqref{eqn:operateurdual}, we have, for every $h \in \mathfrak{H}$ and $\omega \in \Omega$,
\begin{equation}\label{eqn:accroissementisocasdebase}
X_g(\omega + j(h))=y(\omega + j(h))=y(\omega) + y(j(h)) = X_g(\omega)+ \langle g , h \rangle_\mathfrak{H}.
\end{equation}

Now consider a general $g \in \mathfrak{H}$. Choose a sequence $(g_n)_{n \geq 1}$ in $j^*(\Omega^*)$ such that $g_n  \to g$ in $\mathfrak{H}$. Then $(X_{g_n})_{n\geq 1}$ converges to $X_g$ in $L^2(\Omega)$ and, up to extracting a subsequence which we still denote $(X_{g_n})_{n\geq 1}$, the convergence holds almost surely. Define the measurable sets
\begin{align}
	\Omega_{g} &:=\left\{\omega \in \Omega\,:\; \left(X_{g_n}(\omega)\right)_{n\geq 1} \text{ converges in } \R\right\}, \nonumber\\
	\Omega_{g}^{\prime} &:=\left\{\omega \in \Omega\,:\; \lim_{n\to \infty} X_{g_n}(\omega)=X_g(\omega)\right\},
\end{align}
so that $\P(\Omega_{g}^{\prime})=1$ and $\Omega_{g}^{\prime} \subseteq \Omega_{g}$. Next, set
$$
\begin{cases}
	\tilde{X}_g(\omega) &:= \lim_{n\to \infty} X_{g_n}(\omega), \quad \text{ for all } \omega \in \Omega_{g},\\
	\tilde{X}_g(\omega) &:= 0, \quad \text{ for all } \omega \in \Omega\setminus \Omega_{g}.
\end{cases}$$
Since $\tilde{X}_g=X_g$ on $\Omega_g^\prime$, we have $\tilde{X}_g=X_g$ $\P$-a.s. Moreover, \eqref{eqn:accroissementisocasdebase} yields, for every $n \geq 1$, $\omega \in \Omega_{g}$ and $h \in \mathfrak{H}$,
\begin{align*}
	X_{g_n}(\omega + \varepsilon j(h))&=X_{g_n}(\omega) + \langle g_n,h \rangle_{\mathfrak{H}},
\end{align*}
which in turn gives $\Omega_g + j(\mathfrak{H}) \subseteq \Omega_{g}$ and, for every $\omega \in \Omega_{g}$ and $h \in \mathfrak{H}$,
\begin{align*}
	\tilde{X}_g(\omega + j(h))&= \lim_{n\to \infty} X_{g_n}(\omega +  j(h))\\
	&= \lim_{n\to \infty} \big(X_{g_n}(\omega) + \langle g_n,h \rangle_{\mathfrak{H}}\big)\\
	&= \tilde{X}_g(\omega) + \langle g,h \rangle_{\mathfrak{H}}.
\end{align*}
This concludes the proof.
\end{proof}

Since $X_g=\widetilde{X}_g$ almost surely for every $g \in \mathfrak{H}$, the family $\{\widetilde{X}_g \, : \, g \in \mathfrak{H} \}$ is itself an isonormal process on $(\Omega,\mathcal{F},\mathbb{P},\mathfrak{H})$, and one can therefore develop a Malliavin calculus with respect to it. In what follows, we write $\widetilde{D}^q$ for the Malliavin derivative of order $q \in \N$ associated with $\{\widetilde{X}_g \, : \, g \in \mathfrak{H} \}$, $\widetilde{\mathbb{D}}^{m,p}$ for the corresponding Sobolev-type space, and $\widetilde{I}_q$ for the associated $q$th multiple stochastic integral. Within this framework, we can derive a Taylor expansion for stochastic integrals.

\begin{Thm}\label{thm:main}
Let $q \in \N$ and $f \in \mathfrak{H}^{\odot  q}$ be fixed, and set $F=I_q(f)$. There exist two events $\Omega_F^\prime \subseteq \Omega_F$ with $\mathbb{P}(\Omega_F^\prime)=\mathbb{P}(\Omega_F)=1$, together with a random variable $\widetilde{F} \in \widetilde{\mathbb{D}}^{\infty}$, such that
\begin{enumerate}[(a)]
\item  $\Omega_F + j(\mathfrak{H}) \subseteq \Omega_F$;
\item for all $\ell \in \N$, $\widetilde{D}^\ell \widetilde{F}=D^\ell F$ on $\Omega_F^\prime$;
\item for all $\omega \in \Omega_F$, $h \in \mathfrak{H}$ and $0 \leq \ell \leq q$,
\begin{align}\label{eqn:taylorform}
\widetilde{D}^\ell \widetilde{F}(\omega + j(h)) = \sum_{k=0}^{q-\ell} \frac{1}{k!} \langle \widetilde{D}^{k+\ell} \widetilde{F}(\omega),h^{\otimes k} \rangle_{\mathfrak{H}^{\otimes k}}.
\end{align}
\item for all $\ell >q$, $\widetilde{D}^\ell \widetilde{F}=0$.
\end{enumerate}
\end{Thm}

\begin{proof}
Let us first assume that there exists $f_1\dots,f_q \in \mathfrak{H}$ such that 
\begin{equation}\label{eqn:fsymsimple}
f=f_1 \odot \dots  \odot f_q
\end{equation}
It is well-known, see e.g. \cite[Appendix A]{Ayache2025}, that we have
\begin{equation}\label{eqn:factorisationofF}
F = \prod_{\ell=1}^m H_{q_\ell}(X_{\tilde{f}_\ell})
\end{equation}
where $q_1 +\dots+q_m=q$ and, for any $1 \leq \ell \leq m$, $q_\ell$ stands for the multiplicity of $\tilde{f}_\ell$ in $\{f_1,\dots,f_q\}$ and $H_{q_\ell}$ is the $q_\ell$th Hermite polynomial\footnote{For all $n \in \N$, the $n$th Hermite polynomial is the polynomial of degree $n$ denoted  by $H_n$ and defined, for every $x\in\R$, as $H_n(x)=(-1)^n e^{x^2 /2} \frac{d^n}{dx^n} e^{-x^2 /2}$.}.
Let us then consider the events, of probability $1$, $\Omega_{\tilde{f}_1}^\prime,\Omega_{\tilde{f}_1}, \dots,\Omega_{\tilde{f}_m}^\prime,\Omega_{\tilde{f}_m}$ given by Proposition \ref{prop:casdebase} and set
\[ \Omega_F^\prime = \bigcap_{\ell=1}^m \Omega_{\tilde{f}_\ell}^\prime, \quad \Omega_F = \bigcap_{\ell=1}^m \Omega_{\tilde{f}_\ell} \]
and
\begin{equation}\label{eqn:defoftildeF}
\widetilde{F} := \widetilde{I}_q(f)=\prod_{\ell=1}^m H_{q_\ell}(\widetilde{X}_{\tilde{f}_\ell}). 
\end{equation}
It is clear that $\Omega_F^\prime \subseteq \Omega_F$, $\mathbb{P}(\Omega_F^\prime)=\mathbb{P}(\Omega_F)=1$, $\Omega_F + j(\mathfrak{H}) \subseteq \Omega_F$  and, in view of the expressions \eqref{eqn:factorisationofF} and \eqref{eqn:defoftildeF}, for all $0 \leq k \leq q$, $D^kF=\widetilde{D}^k\widetilde{F}$ on $\Omega_F^\prime$. It remains to show that \eqref{eqn:taylorform} holds true for any $\omega \in \Omega_F $ and $h \in \mathfrak{H}$. For any such $\omega$ and $h$, we know from Proposition \ref{prop:casdebase} and the definition \eqref{eqn:defoftildeF} of $\widetilde{F}$, that
\[\widetilde{F}(\omega+j(h))=\prod_{\ell=1}^m H_{q_\ell}(\widetilde{X}_{\tilde{f}_\ell}(\omega)+\langle \tilde{f}_\ell, h \rangle_\mathfrak{H} ). \]
Next, for any $1 \leq \ell \leq m$, we apply the Taylor formula of the polynomial $H_{q_\ell}$ to get
\begin{align*}
\widetilde{F}(\omega+j(h))& =\prod_{\ell=1}^m \left( \sum_{k_\ell =0}^{q_\ell} \frac{1}{k_\ell !} H_{q_\ell}^{(k_\ell)}(\widetilde{X}_{\tilde{f}_\ell}(\omega)) \left(\langle \tilde{f}_\ell, h \rangle_\mathfrak{H}\right)^{k_\ell} \right) \\
&= \prod_{\ell=1}^m \left( \sum_{k_\ell =0}^{q_\ell} \frac{1}{k_\ell !}\left\langle \left(\widetilde{D}^{k_\ell} H_{q_\ell}(\widetilde{X}_{\tilde{f}_\ell})\right)(\omega) , h^{\otimes {k_\ell}} \right\rangle_{\mathfrak{H}^{\otimes {k_\ell}}} \right) \\
& = \sum_{k=0}^q \frac{1}{k!}\left( \sum_{\substack{ k_1,\dots,k_m\\ k_1+\dots+k_m=k}} k! \prod_{\ell=1}^m \frac{1}{k_\ell !}\left\langle \left(\widetilde{D}^{k_\ell} H_{q_\ell}(\widetilde{X}_{\tilde{f}_\ell})\right)(\omega) , h^{\otimes {k_\ell}} \right\rangle_{\mathfrak{H}^{\otimes {k_\ell}}} \right)
\end{align*}

A straightforward induction on $m$ combined with the Leibniz rule for the Malliavin derivative \cite[Exercise 2.3.10]{nourdin2012normal} then yields, for all $0 \leq k \leq q$, $\omega \in \Omega_F$ and $h \in \mathfrak{H}$,
\begin{align*}
\sum_{\substack{ k_1,\dots,k_m\\ k_1+\dots+k_m=k}} &  k!\prod_{\ell=1}^m \frac{1}{k_\ell !}\left\langle \left(\widetilde{D}^{k_\ell} H_{q_\ell}(\widetilde{X}_{\tilde{f}_\ell})\right)(\omega) , h^{\otimes {k_\ell}} \right\rangle_{\mathfrak{H}^{\otimes {k_\ell}}} \\
& = \left\langle \widetilde{D}^k \left(\prod_{\ell=1}^m H_{q_\ell}(\widetilde{X}_{\tilde{f}_\ell}) \right)(\omega),h^{\otimes k} \right\rangle_{\mathfrak{H}^{\otimes k}}\\
& = \langle \widetilde{D}^k \widetilde{F}(\omega),h^{\otimes k} \rangle_{\mathfrak{H}^{\otimes k}}
\end{align*}

Formula \eqref{eqn:taylorform} for $\ell \geq 1$ is proved in exactly the same way.

Now, we assume $f$ that is a linear combination
\begin{equation}\label{eqn:casesoflincomb}
 f := \sum_{\ell=1}^m \lambda_\ell f_\ell,
\end{equation}
where, for all $1 \leq \ell \leq m$, $\lambda_\ell \in \R$ and $f_\ell \in \mathfrak{H}^{\odot  q}$ has the simple form \eqref{eqn:fsymsimple}. Then, setting, for any such $\ell$, $F_\ell = I_q(f_\ell)$, by linearity of the Malliavin derivatives and the stochastic integral, it suffices to take $\Omega_F^\prime= \bigcap_{\ell=1}^m \Omega_{F_\ell}^\prime$, $\Omega_F= \bigcap_{\ell=1}^m \Omega_{F_\ell}$ and $\widetilde{F}=\widetilde{I}_q(f)$.

Finally, if $F=I_q(f)$ for a general $f \in \mathfrak{H}^{\odot  q}$, there exists a sequence $(f_n)_{n \geq 1} \subseteq \mathfrak{H}^{\odot  q}$ which converges to $f$ in $\mathfrak{H}^{\otimes q}$ such that, for any $n \geq 1$, $f_n$ has form \eqref{eqn:casesoflincomb}. In this case the definition of the Malliavin derivative yields a subsequence, that we still denote $(f_n)_{n \geq 1}$, such that, for any $\ell \in \N$, $\widetilde{D}^\ell \widetilde{I}_q(f_n)  \to \widetilde{D}^\ell \widetilde{I}_q(f)$ and $D^\ell I_q(f_n)  \to D^\ell F$, almost surely. Let us set, for any $n \geq 1$, $F_n := I_q(f_n)$, $\widetilde{F}_n := \widetilde{I}_q(f_n)$,
\begin{align}
	\Omega_{1} &:=\left\{\omega \in \Omega\,:\; \forall 0 \leq \ell \leq q, \, \left(\widetilde{D}^\ell \widetilde{F}_n(\omega)\right)_{n\geq 1} \text{ converges in } \mathfrak{H}^{\otimes \ell}\right\}, \nonumber\\
	\Omega_{1}^{\prime} &:=\left\{\omega \in \Omega\,:\; \forall \ell \in \N, \, \lim_{n\to \infty} \widetilde{D}^\ell \widetilde{F}_n(\omega)=\widetilde{D}^\ell \widetilde{I}_q(f)(\omega) \text { and }\right.\\
	&\left.\hskip5cm \lim_{n\to \infty} D^\ell F_n(\omega)=\widetilde{D}^\ell F(\omega) \right\}, \nonumber
\end{align}
\[ \Omega_F^\prime = \Omega_1^\prime \cap \bigcap_{n \geq 1} \Omega_{F_n}^\prime, \quad  \Omega_F= \Omega_1 \cap \bigcap_{n \geq 1} \Omega_{F_n} \]
and, for $0 \leq \ell \leq q$ the random variables
$$
\begin{cases}
	\widetilde{F}_\ell(\omega) &:= \lim_{n\to \infty} \widetilde{D}^\ell\widetilde{F}_n(\omega), \quad \text{ for all } \omega \in \Omega_F,\\
	\widetilde{F}_\ell(\omega) &:= 0, \quad \text{ for all } \omega \in \Omega\setminus \Omega_F.
\end{cases}$$
Clearly $\Omega_F^\prime \subseteq \Omega_F$ and $\mathbb{P}(\Omega_F^\prime)=\mathbb{P}(\Omega_F)=1$. We now show that one can take $\widetilde{F}=\widetilde{F}_0$. Since $\mathbb{P}(\Omega_1^\prime)=1$, we obtain $\widetilde{F} \in \widetilde{\mathbb{D}}^\infty$, with $\widetilde{D}^\ell\widetilde{F}=\widetilde{F}_\ell$ for $0 \leq \ell \leq q$ and $\widetilde{D}^\ell\widetilde{F}=0$ for $\ell >q$. Moreover, for every $\omega \in \Omega_F^\prime$ and every $\ell \in \N$,
\[ \widetilde{D}^\ell\widetilde{F}(\omega)=\lim_{n \to + \infty}\widetilde{D}^\ell \widetilde{I}_q(f_n)(\omega)= \lim_{n \to + \infty}D^\ell F_n(\omega) = D^\ell F(\omega).  \]

To conclude, we note that the previous case of elements of the form \eqref{eqn:casesoflincomb} gives, for every $\omega \in \Omega_F$, $h \in \mathfrak{H}$, $n \geq 1$ and $0 \leq \ell \leq q$,
\[ \widetilde{D}^\ell \widetilde{I}_q(f_n)(\omega + j(h)) = \sum_{k=0}^{q-\ell} \frac{1}{k!} \langle \widetilde{D}^{k+\ell}\widetilde{I}_q(f_n)(\omega),h^{\otimes k} \rangle_{\mathfrak{H}^{\otimes k}}.\]
 It follows that $\omega + j(h) \in \Omega_F$, and hence $\Omega_F + j(\mathfrak{H}) \subseteq \Omega_F$. For such $\omega$ and $h$, we also get, for all $0 \leq \ell \leq q$,
  \begin{align*}
  \widetilde{D}^\ell\widetilde{F}(\omega + j(h))=\widetilde{F}_\ell(\omega + j(h)) &= \sum_{k=0}^{q-\ell} \frac{1}{k!} \langle \widetilde{F}_{\ell+k}(\omega),h^{\otimes k} \rangle_{\mathfrak{H}^{\otimes k}}\\ &=\sum_{k=0}^{q-\ell} \frac{1}{k!} \langle \widetilde{D}^{\ell+k}\widetilde{F}(\omega),h^{\otimes k} \rangle_{\mathfrak{H}^{\otimes k}}.
  \end{align*}
\end{proof}

For a random variable $F=I_q(f)$ with $f \in \mathfrak{H}^{\odot  q}$, the events $\Omega_F$ and $\Omega_F'$ appearing in Theorem \ref{thm:main} have the advantage of being independent of $h \in \mathfrak{H}$; however, the associated Taylor expansion \eqref{eqn:taylorform} holds only for a random variable $\widetilde{F}$ almost surely equal to $F$. If instead we fix the vector $h \in \mathfrak{H}$ beforehand (which is enough to verify the \textbf{(SGD)} condition) then the Taylor expansion can be written directly for $F$.

\begin{Cor}\label{cor:taylor}
Let $q \in \N$ and $f \in \mathfrak{H}^{\odot  q}$ be fixed, and set $F=I_q(f)$. For every $h \in \mathfrak{H}$, there exists an event $\Omega_h$ with $\mathbb{P}(\Omega_h)=1$ such that, for every $\omega \in \Omega_h$ and $0 \leq \ell \leq q$,
\[D^\ell F(\omega + j(h)) = \sum_{k=0}^{q-\ell} \frac{1}{k!} \langle D^{k+\ell} F(\omega),h^{\otimes k} \rangle_{\mathfrak{H}^{\otimes k}}. \]
\end{Cor}

\begin{proof}
Let $\Omega_F^\prime$ and $\Omega_F$ be the events given by Theorem \ref{thm:main}. Fix $h \in \mathfrak{H}$ and consider the translation $T_h \, : \, \Omega \to \Omega \, : \, \omega \mapsto \omega + j(h)$, together with the induced pushforward measure
\[ \mathbb{P}_h = (T_h)_\ast \mathbb{P}. \]
The Cameron-Martin theorem \cite{Cameron1944} ensures that $\mathbb{P}_h$ is equivalent to $\mathbb{P}$, with Radon-Nikodym derivative
\begin{equation}\label{eq:CMLiklhood}
 \frac{ d \mathbb{P}_h}{d \mathbb{P}}= \exp \left( X_h - \frac{1}{2} \|h \|_{\mathfrak{H}}^2 \right).
 \end{equation}
 Consequently, the event
 \[ \{ \omega \, : \, \omega + j(h) \in \Omega_F^\prime \}=\Omega_F^\prime-j(h) \]
 has probability $1$. The set $\Omega_h:=(\Omega_F^\prime-j(h)) \cap \Omega_F^\prime \cap \Omega_F$ is therefore of probability $1$, and the conclusion follows from \eqref{eqn:taylorform} combined with the identity $\widetilde{D}^\ell \widetilde{F}=D^\ell F$, valid on $\Omega_F^\prime$ for every $\ell \in \N$.
\end{proof}

\subsection{Towards (SGD): construction of the processes $S$ and $\Gamma$}\label{sub:malcalsdes}

Combining the results of Subsection \ref{sub:taylor} with Proposition \ref{pro:Nualart}, we are now in a position to construct the processes $\{S_{t,\varepsilon}^h \, : \, t \in [0,T], \, \varepsilon \in \R, \, h \in \mathfrak{H} \}$ and $\{\Gamma_{t,\varepsilon}^h \, : \, t \in [0,T], \, \varepsilon \in \R, \, h \in \mathfrak{H} \}$ announced at the end of Subsection \ref{sub:malialter}. In what follows, given $t \in \R_+$, $h \in \mathfrak{H}$ and $1 \leq k \leq q$, the notation $\left<D^kF_t,h^{\otimes k}\right>_{\HHH^{\otimes k}}$ stands for
 \begin{align*}
	& \left(\left<D^kF^1_t),h^{\otimes k}\right>_{\HHH^{\otimes k}}, \ldots, \left<D^kF^d_t,h^{\otimes k}\right>_{\HHH^{\otimes k}}\right),
	\end{align*}
	while we set $\left<D^0F_t,h^{\otimes 0}\right>_{\HHH^{\otimes 0}}:= F_t$.
	
\begin{Prop}\label{prop:Version}
	Let $F=\{F_t=(I_q(f_t^1),\ldots,I_q(f_t^m))\,,\;t\in \R_+\}$ be an $\R^m$-valued chaos process on an abstract Wiener space $(\Omega,\mathcal{F},\P,\mathfrak{H})$, and assume that $F$ satisfies hypothesis $\mathbf{(H_2)}$ with $H\in(\tfrac{1}{2},1)$.
	Let $R=\{R^h_{t,\varepsilon}\,;\;t\in\R,\, \varepsilon\in\R,\, h\in\HHH\}$ and $S=\{S^{h}_{t,\varepsilon}\,;\;t\in\R,\, \varepsilon\in\R,\, h\in\HHH\}$ be the $\R^m$-valued stochastic processes defined, respectively, by
 \begin{align}\label{eq:defRS}
	R_{t,\varepsilon}^h=F_t(\bullet + \varepsilon\,j(h)) \quad \text{ and } \quad S^{h}_{t,\varepsilon} = \sum_{k=0}^{q}  \frac{\varepsilon^k}{k!}\langle D^k F_t,h^{\otimes k} \rangle_{\mathfrak{H}^{\otimes k}}.
 \end{align}
Then the following hold:
\begin{enumerate}[(a)]
	\item $S$ is a version of $R$;
	\item for every $h\in \HHH$, almost every $\omega\in\Omega$ and every $\varepsilon\in\R$, one has $S^{h}_{\star,\varepsilon}(\omega)\in W_{2}^{1-\alpha}\left(0, T ; \mathbb{R}^{m}\right)$;
	\item for every $\beta\in(0,H-\tfrac{1}{2})$ and $1-\alpha = H-\beta$, and for almost every $\omega\in\Omega$, the map from $\R$ into $W_{2}^{1-\alpha}\left(0, T ; \mathbb{R}^{m}\right)$ defined by $\varepsilon\mapsto S^{h}_{\star,\varepsilon}(\omega)$ is of class $\mathcal{C}^{\infty}$. Moreover, for every $h\in \HHH$, almost every $\omega\in\Omega$, every $\varepsilon\in\R$ and every $\ell \in \N$,
	$$
\frac{d^\ell}{d\varepsilon^\ell} S^{h}_{\star,\varepsilon}(\omega)= \left\{
\begin{array}{lcl}
\sum_{k=l}^{q}\frac{\varepsilon^{k-l}}{k!} \left<D^kF_\star(\omega) ,h^{\otimes k}\right>_{\HHH^{\otimes k}}, & \text{if }& 0 \leq \ell \leq q\\
0,&  \text{if }& \ell >q.
\end{array}\right.
$$
\end{enumerate}	 
\end{Prop}
\begin{proof}
Point (a) is a direct consequence of Corollary \ref{cor:taylor}. Point (b) follows from hypothesis $\mathbf{(H_2)}$ (see the discussion following inequality \eqref{eq:HolderF0} above) together with the embedding \eqref{eq:injection2}. Point (c) is a straightforward consequence of elementary differential calculus.
\end{proof}

Building on this proposition, we define, for each $h \in \mathfrak{H}$,
\[ \Phi_h \, : \, \R\times \Omega \to W_{2}^{1-\alpha}\left(0, T ; \mathbb{R}^{m}\right) \, : \, (\epsilon,\omega) \mapsto S^{h}_{\star,\varepsilon}(\omega).\]
With these notations in hand, we set, for all $h \in \mathfrak{H}$, $t \in [0,T]$, $\varepsilon \in \R$ and $\omega \in \Omega$,
\begin{equation}\label{defofGamma}
\Gamma_{h,t}(\varepsilon,\omega) := \Psi_t(\Phi_h(\varepsilon,\omega)),
\end{equation}
where $\Psi$ is the mapping introduced in \eqref{eqn:frechetdifferentiabilityPSI}. Recall also that $\mathcal{D}\Psi(\varphi)[\psi]$ denotes the Fréchet derivative of $\Psi$ at $\varphi$ in the direction $\psi$.

We now verify that the process $\{\Gamma_{h,t}(\varepsilon,\bullet) \, : \, t \in [0,T], \, \varepsilon \in \R, \, h \in \mathfrak{H} \}$ enjoys the required regularity properties.

\begin{Prop}\label{lem:Gamma}
Let $F=\{F_t=(I_p(f_t^1),\ldots,I_p(f_t^m))\,,\;t\in \R_+\}$ be an $\R^m$-valued chaos process on an abstract Wiener space $(\Omega,\mathcal{F},\P,\mathfrak{H})$, satisfying hypothesis $\mathbf{(H_2)}$, and assume that the coefficients $(b_k)_{1 \leq k \leq d}$ and $(\sigma_{k,\ell})_{1 \leq k \leq d,\, 1 \leq \ell \leq m}$ of the stochastic differential equation \eqref{eq:SDEsDrivenF} satisfy hypothesis $\mathbf{(H_1)}$.
Then, for every $h\in\HHH$, almost every $\omega\in\Omega$, every $t\in [0,T]$ and every $1 \leq k \leq d$, the map $\Gamma_{h,t}^{k}(\star,\omega)$ is differentiable on $\R$, with derivative given, for every $\varepsilon \in \R$, by
	\begin{equation}
		\frac{d}{d\varepsilon}\Gamma_{h,t}^{k} (\varepsilon,\omega)= \mathcal{D} \Psi_t^k(S_{\star, \varepsilon}^{h}(\omega))\left[\frac{d}{d \varepsilon} S_{\star, \varepsilon}^{h}(\omega)\right].
	\end{equation}
In particular, for $\varepsilon=0$, we get
	\begin{equation}\label{eq:DerivativeOfXt}
		\left.\frac{d}{d\varepsilon}\Gamma_{h,t}^{k} (\varepsilon,\omega)\right|_{\varepsilon=0}= \mathcal{D} \Psi_t^k(F_\star(\omega))\left[\langle D F_\star(\omega), h\rangle_{\HHH}\right].
	\end{equation}
\end{Prop}
\begin{proof}
Since $F$ satisfies hypothesis $\mathbf{(H_2)}$, point (c) of Proposition \ref{prop:Version} ensures that, for each $h\in\HHH$ and almost every $\omega\in\Omega$, the mapping $\Phi_{h}(\star,\omega)$ is of class $\mathcal{C}^\infty$. On the other hand, Proposition \ref{pro:Nualart} guarantees that $\Psi_t^k$ is Fréchet differentiable for every $t\in [0,T]$ and every $1 \leq k \leq d$. The chain rule then yields, for all $h\in\HHH$, almost every $\omega\in\Omega$, every $t\in [0,T]$ and every $1 \leq k \leq d$, the differentiability of $\Gamma^k_{h,t}(\star,\omega)$. Moreover, under the same assumptions and for any $\varepsilon\in\R$, the chain rule gives
\begin{equation}
\begin{aligned}
 \frac{d}{d \varepsilon} \Gamma_{h,t}^{k} (\varepsilon,\omega)  
& =\mathcal{D} (\Psi_t^k\circ\Phi_h(\star,\omega))[\varepsilon] \\
& =\mathcal{D} \Psi_t^k\left(\Phi_h(\varepsilon,\omega)\right)[\mathcal{D} \Phi_h(\star,\omega)[\varepsilon]]\\
& =\mathcal{D}  \Psi_t^k\left(S_{\star, \varepsilon}^{h}(\omega)\right)\left[\frac{d}{d \varepsilon} S_{\star, \varepsilon}^{h}(\omega)\right].
\end{aligned}
\end{equation}
This concludes the proof of the first part of Proposition \ref{lem:Gamma}. Equation \eqref{eq:DerivativeOfXt} is then a direct consequence of \eqref{defofGamma} and point (c) of Proposition \ref{prop:Version}.
\end{proof}

\subsection{Towards (RAC): a refined version of the driving chaotic process}\label{sub:version}

In this subsection, we construct a version of the driving chaotic process that enjoys strong continuity and smoothness properties. More precisely, we establish the following theorem.

\begin{Thm}\label{thm:ray}
	Let $F$ be a real-valued chaos process on an abstract Wiener space $(\Omega,\mathcal{F},\P,\mathfrak{H})$.  
	Assume that $F$ verifies the hypothesis $\mathbf{(H_2)}$, with $H\in (\frac{1}{2},1)$. 
	For any \(h\in\HHH\), there exists a stochastic process $(\tilde F^{\,h}_t)_{t\in[0,T]}$ such that:
	\begin{enumerate}[(a)]
	\item For any $t\in [0,T]$,  $\tilde{F}^{\,h}_t=F_t$ $\P$-a.s.;
	\item For all $\omega\in\Omega$, the map
	$
	\R_+\times\R \ni (t,\varepsilon)\longmapsto \widetilde F^{\,h}_t\big(\omega+\varepsilon\,j(h)\big)
	$
	is continuous.
	\item Let $\beta\in(0,H-\frac{1}{2})$ and $1-\alpha = H-\beta$. For all $\omega\in\Omega$, the mapping from $\R$ into $W_{2}^{1-\alpha}\left(0, T ; \mathbb{R}^{m}\right)$ defined by $\varepsilon\mapsto \tilde{F}^{h}_{\star}(\omega + \varepsilon j(h))$ is of class $\mathcal{C}^{\infty}$ on $\R$.
	\end{enumerate}
\end{Thm}

The proof of Theorem \ref{thm:ray} relies on the following lemma. As its proof is somewhat tangential to the main thread of this paper, we have chosen to postpone it to Appendix \ref{appendix:proofofLemma}.

\begin{Lemma}\label{indepnt}
	Let $q$ be a positive integer, $e_0\in \HHH$, with $\|e_0\|_{\HHH}=1$, and  $(f_t)_{t\in[0,T]}$ be a family of vectors in  $\HHH ^{\odot q}$.   
	We denote by $V$ the orthogonal complement of $\operatorname{span}\{e_0\}$ in $\HHH$. 
	Hence, there exist $(f_{t,q-k})_{t\in [0,T]}\subset V^{\odot  {(q-k)}}$, with $0 \leq k \leq q$, such that
	\begin{equation}\label{indep}
		I_q(f_t)= \sum _{k=0}^{q} I_{k} (e_0^{\otimes k}) I_{q-k} (f_{t,q-k}),
	\end{equation}
	where $I_{0} (h^{\otimes 0})=1$. Furtheremore, for each $0 \leq k \leq q$, there exists a positive constant $C$ such that for any $s,t\in [0,T]$,
	\begin{equation}\label{eq:ftn-k-holder}
		\|f_{t,q-k}-f_{s,q-k}\|_{\HHH^{\otimes {(q-k)}}}\leq C \|f_t-f_s\|_{\HHH^{\otimes q}}.
	\end{equation}
\end{Lemma}

We now have all the ingredients needed to prove Theorem \ref{thm:ray}.

\begin{proof}[Proof of Theorem \ref{thm:ray}]
	Let $h\in \HHH$ be fixed.  The case \(h=0\) is trivial, so we assume \(h\neq 0\) and set \(e_0:=h/\norm{h}_{\HHH}\).
	%We split the proof into three steps.
%\medskip
	%\noindent\textbf{1) Step 1.}
	Let $V$ be the orthogonal complement of $\operatorname{span}\{e_0\}$ in $\HHH$ and $(e_n)_{n\geq 1}$ an orthonormal basis of $V$. 
	Then $(e_n)_{n\geq 0}$ is an orthonormal basis of $\HHH$.
	Let $X=\{X_g\,: g\in \HHH\}$ denote the isonormal Gaussian process on the abstract Wiener space $(\Omega,\mathcal{F},\P,\mathfrak{H})$ introduced in Subsection \ref{subsec:abstractwienerspaces}, and let $\tilde{X}=\{\tilde{X}_g\,:\;g\in \HHH\}$ be its version provided by Proposition \ref{prop:casdebase}, whose notations we adopt here. Consider first the event
\[ \Omega_h = \bigcap_{n \geq 1} \Omega_{g_n}.\]
One clearly has $\P(\Omega_{h})=1$ and $\Omega_{h} + \R j(h)\subset \Omega_{h}$, where
\[ \Omega_{h} + \R j(h):=\{\omega + \varepsilon j(h)\, : \, \omega\in \Omega_{h},\, \varepsilon\in \R\}.\]
Moreover, by equation \eqref{eqn:fundamental}, for every $\omega \in \Omega_{h}$ and $\varepsilon\in \R$,
		\begin{align}\label{eq:xi}
			\tilde{X}_{e_0}(\omega+\varepsilon j(h))=\tilde{X}_{e_0}(\omega)+\varepsilon\norm{h}_{\HHH}^2, \quad \text{ and } \quad \tilde{X}_{e_n}(\omega+\varepsilon j(h))=\tilde{X}_{e_n}(\omega),
		\end{align}
		for  $n\ge 1$.

Since Proposition \ref{prop:casdebase} ensures that $\tilde{X}_g=X_g$ $\P$-a.s.\ for every $g \in \HHH$, we immediately obtain, for every $t\in [0,T]$,
\begin{align}\label{eq:Ft=It}
	F_t=I_q(f_t)=\tilde{I}_q(f_t) \quad\P\text{-a.s.}
\end{align}
	Furthermore, Lemma \ref{indepnt} yields the existence of families $(f_{t,q-k})_{t\in [0,T]}\subset V^{\odot  {(q-k)}}$, for $k=0,\ldots, q$, such that, for every $t\in [0,T]$,
	\begin{equation}\label{indep0}
		\tilde{I}^h_q(f_t)= \sum _{k=0}^{q}  \tilde{I}_{k}(e_0^{\otimes k}) (f_{t,q-k}) \quad\P\text{-a.s.}
	\end{equation}
	Consider now the $\sigma$-algebra
	$$
	\mathcal{G}:=\sigma(\tilde{X}_{e_n}:\,n\ge 1).
	$$
	For every $t \in [0,T]$ and $0 \leq k \leq q$, since $f_{t,q-k} \in V^{\odot  {(q-k)}}$ and $(e_n)_{n\geq 1}$ is an orthonormal basis of $V$, there exists a sequence of polynomials $(P_{n_\ell})_{\ell \in \N}$ such that $(P_{n_\ell}(\tilde{X}_{e_1},\dots,\tilde{X}_{e_{n_\ell}}))_{\ell \in \N}$ converges almost surely to $\tilde{I}_{q-k} (f_{t,q-k})$; see for instance the proof of \cite[Lemma 4.3]{Nourdin2013}.
	Introduce the event
	\[ \Omega_{h,t} := \{ \omega \in \Omega \, : \, (P_{n_\ell}(\tilde{X}_{e_1}(\omega),\dots,\tilde{X}_{e_{n_\ell}}(\omega)))_{\ell \in \N} \text{ converges}\}.\]
	It is clear that $\Omega_{h,t,k} \in \mathcal{G}$ and $\mathbb{P}(\Omega_{h,t})=1$. Accordingly, the random variable
	\[ Y^{h}_{t,k} \, : \, \omega \mapsto \begin{cases}
	\lim_{\ell \to + \infty} P_{n_\ell}(\tilde{X}_{e_1}(\omega),\dots,\tilde{X}_{e_{n_\ell}}(\omega)) & \text{ if } \omega \in \Omega_{h,t,k} \\
	0 & \text{ otherwise}
	\end{cases}\]
	is $\mathcal{G}$-measurable and almost surely equal to $\tilde{I}_{q-k} (f_{t,q-k})$.
	Combining the hypercontractivity property for multiple Wiener--It\^o integrals \eqref{eq:hyper}, the isometry formula \eqref{eq:isom}, inequality \eqref{eq:ftn-k-holder} and Hypothesis $\mathbf{(H_2)}$, we deduce that, for every $0 \leq k \leq q$, $p \geq 2$ and $s,t\in [0,T]$,
	\begin{align*}
	\E&[|Y^{h}_{t,k} - Y^h_{s,k}|^p] \\
		&=\E[|\tilde{I}^h_{q-k} (f_{t,q-k}) - \tilde{I}^h_{q-k} (f_{s,q-k})|^p] \\
		&\leq (p-1)^{p(q-k)/2}\big\{\E[|\tilde{I}^h_{q-k} (f_{t,q-k}) - \tilde{I}^h_{q-k} (f_{s,q-k})|^2]\big\}^{p/2}\\
		&=(p-1)^{p(q-k)/2}\big((q-k)!\big)^{p/2} \|f_{t,q-k}-f_{s,q-k}\|_{\HHH^{\otimes {q-k}}}^p\\
		&\leq C(p-1)^{p(q-k)/2}\big((q-k)!\big)^{p/2} \|f_{t}-f_{s}\|_{\HHH^{\otimes q}}^p\\
		&\leq \tilde{C}(p-1)^{p(q-k)/2}\big((q-k)!\big)^{p/2} |t-s|^{pH}.
	\end{align*}
	Hence, for any $\beta \in (0,H-\frac{1}{2})$, the Kolmogorov continuity theorem applied on the abstract Wiener space $(\Omega, \mathcal{G}, \P, V)$ yields, for each $0 \leq k \leq q$, a $\mathcal{G}$-measurable stochastic process $(\tilde{Y}^{h}_{t,k})_{t\in [0,T]}$ with $(H-\beta)$-Hölder continuous sample paths and such that, for every $t\in [0,T]$,
	\begin{align}\label{eq:Jt=It}
		\tilde{Y}^{h}_{t,k}=\tilde{I}^h_{q-k} (f_{t,q-k})\quad \P\text{-a.s.}
	\end{align}
	We are now ready to introduce the stochastic process $(\tilde F^{\,h}_t)_{t\in[0,T]}$ defined, for every $t \in [0,T]$, by
	\begin{equation}\label{eq:defFt}
		\tilde{F}^{\,h}_t:= \sum _{k=0}^{q} \tilde{I}_{k}(e_0^{\otimes k}) \tilde{Y}^{h}_{t,k} \mathbbm{1}_{\Omega_{h}}.
	\end{equation}
It remains to check that it enjoys the announced properties.

First, combining \eqref{eq:Ft=It}, \eqref{indep0}, \eqref{eq:Jt=It} and \eqref{eq:defFt} with $\mathbb{P}(\Omega_h)=1$, we see that $\tilde{F}^{\,h}_t=F_t$ $\P$-a.s.\ for every $t\in [0,T]$.

Secondly, the $\mathcal{G}$-measurability of $(\tilde{Y}^{h}_{t,k})_{t \in [0,T]}$, together with \eqref{eq:xi}, yields, for every $0 \leq k \leq q$, $t\in [0,T]$, $\varepsilon\in \R$ and $\omega\in \Omega_{h}$,
\begin{align*}
	 \tilde{Y}^{h}_{t,k}(\omega+ \varepsilon j(h))=	 \tilde{Y}^{h}_{t,k}(\omega),
\end{align*}
while the first part of the proof of Theorem \ref{thm:main}, combined with equality \eqref{eqn:derivativeofintegral}, gives, for the same $\varepsilon$ and $\omega$,
\[ \tilde{I}_{k}(e_0^{\otimes k}) = \sum_{\ell=0}^q \frac{\varepsilon^\ell}{\ell!} \langle \tilde{I}_{k-\ell}(e_0^{\otimes k})(\omega),h^{\otimes \ell} \rangle_{\mathfrak{H}^{\otimes \ell}} = \sum_{\ell=0}^q \frac{\varepsilon^\ell}{\ell!} \|h\|_\mathfrak{H}^\ell  \tilde{I}_{k-\ell}(e_0^{\otimes {k-\ell}})(\omega).\]
Altogether, since $\Omega_{h} + \R j(h)\subset \Omega_{h}$ and $\tilde{Y}^{h}_{t,k}$ vanishes outside $\Omega_{h}$ for every $t \in [0,T]$ and $0 \leq k \leq q$, we obtain, for every $\omega \in \Omega$, $t \in [0,T]$ and $\varepsilon \in \R$,
\begin{equation}\label{eq:defFt0}
	\tilde{F}^{\,h}_t(\omega+\varepsilon j(h))= \sum _{k=0}^{q} \left( \sum_{\ell=0}^q \frac{\varepsilon^\ell}{\ell!} \|h\|_\mathfrak{H}^\ell  \tilde{I}_{k-\ell}(e_0^{\otimes {k-\ell}})(\omega)\right)\tilde{Y}^{h}_{t,k}(\omega).
\end{equation}
This formula, together with the $(H-\beta)$-Hölder continuity of the sample paths of $(\tilde{Y}^{h}_{t,k})_{t\in [0,T]}$ for every $0 \leq k \leq q$, entails that, for every $\omega\in \Omega$, the map $(t,\varepsilon)\mapsto \tilde{F}^{\,h}_t(\omega+\varepsilon j(h))$ is continuous, and the map from $\R$ into $W_{2}^{1-\alpha}\left(0, T ; \mathbb{R}^{m}\right)$ defined by $\varepsilon\mapsto \tilde{F}^{h}_{\star}(\omega + \varepsilon j(h))$ is of class $\mathcal{C}^{\infty}$ on $\R$.
This completes the proof of Theorem \ref{thm:ray}.
\end{proof}

\section{Malliavin smoothness of the solution} \label{sec:smoothness}

The tools built in Section \ref{sec:tools} bring together all the quantities announced in the roadmap at the end of Subsection \ref{sub:malialter}. We can now prove that the solution of the stochastic differential equation \eqref{eq:SDEsDrivenF} is smooth in the Malliavin sense.

\begin{Thm}\label{thm:MallDiff}
	Let $F=\{F_t=(I_q(f_t^1),\ldots,I_q(f_t^m))\,,\;t\in \R_+\}$ be an $\R^m$-valued chaos process on an abstract Wiener space $(\Omega,\mathcal{F},\P,\mathfrak{H})$.  Assume that, $F$ verifies the hypotheses $\mathbf{(H_2)}$ with $H\in (\frac{1}{2},1)$. Let us also assume  that the coefficients $(b_k)_{1 \leq k \leq d}$ and $(\sigma_{k,\ell})_{1 \leq k \leq d, 1 \leq \ell \leq m}$ of the stochastic differential equation \eqref{eq:SDEsDrivenF} satisfy  the hypotheses $\mathbf{(H_1)}$.
	Then, the unique solution to \eqref{eq:SDEsDrivenF}, $(X_t)_{t \in [0,T]}$, is such that for any $t\in(0,T]$, $X_t\in (\D^{1,\infty})^d$, and, for any $h\in\HHH$ and $1 \leq k \leq d$, we have
	\begin{align}\label{eq:MallDiffXt}
		\left<D X_t^k, h\right>_{\HHH}=  \sum_{\ell=1}^{m} \int_{0}^{t} \Theta_{t}^{k,\ell}(s) \mathrm{d} \left<DF^{\ell}(s), h\right>_{\HHH},
	\end{align}
where 	for $1 \leq k \leq d$, $1 \leq \ell \leq m, \, 0 \leq s \leq t \leq T, s \mapsto \Theta_{t}^{k,\ell}(s)$ satisfies
\begin{align}
	\Theta_{t}^{k,\ell}(s)=\sigma_{k,\ell}\left(X_s\right) &+ \sum_{p=1}^{d} \int_{s}^{t} \partial_{p} b_{k}\left(X_u\right) \Theta_{u}^{p, \ell}(s) \mathrm{d} u \nonumber \\
	&+ \sum_{p=1}^{d} \sum_{q=1}^{m} \int_{s}^{t} \partial_{p} \sigma_{k,q}\left(X_u\right) \Theta_{u}^{p, \ell}(s) \mathrm{d} F^{q}_u,\label{eq:Theta}
\end{align}
and $\Theta_{t}^{k, \ell}(s)=0$ if $s>t$. 
\end{Thm}

\begin{proof}
Fix $t\in [0,T]$ and $k\in\{1,\ldots,d\}$. By Theorem \ref{thm:KusuokaStroock=Shigekawa}, it suffices to show that $X_t^k\in \hat{\D}^{1,\infty}$, and this reduces to verifying that $X_t^k$ satisfies both \textbf{(RAC)} and \textbf{(SGD)}.

We begin with \textbf{(SGD)}. Let $S=\{S^{h}_{s,\varepsilon}\,:\;s\in[0,T],  \varepsilon\in\R, \, h\in\HHH\}$ denote the version of
 $R=\{R^h_{s,\varepsilon}= F(s,\bullet + \varepsilon\,j(h))\,;\;s\in[0,T],  \varepsilon\in\R, \, h\in\HHH\}$ described in \eqref{eq:defRS}. For every fixed $h\in\HHH$, $\varepsilon\in\R$ and $\omega\in\Omega$, the maps $s\mapsto S^{h}_{s,\varepsilon}(\omega)$ and $s\mapsto R^h_{s,\varepsilon}(\omega)$ are continuous, so that, for every fixed $h\in\HHH$ and $\varepsilon\in\R$, and for almost every $\omega\in\Omega$,
 \begin{equation}\label{eq:YEqualZ}
 R^h_{s,\varepsilon}(\omega)=S^{h}_{s,\varepsilon}(\omega), \qquad \text{ for all } s\in[0,T].
 \end{equation}
Consequently, for every fixed $h\in\HHH$ and $\varepsilon\in\R$, and for almost every $\omega\in\Omega$,
 \begin{equation}\label{eq:XEqualPsiZ}
 X_t(\omega + \varepsilon\,j(h))=\Gamma_{h,t}(\varepsilon,\omega).
\end{equation}
 %where $\Gamma_{h,t}^2$ is described in Lemma \ref{lem:Gamma}. 
Fix now $h\in\HHH$, $\varepsilon\in\R$ and $\lambda>0$. Using \eqref{eq:XEqualPsiZ}, we write
 \begin{align}\label{eq:ProbaCv}
 &\P\left[\left\|\frac{1}{\varepsilon}(X_t(\bullet + \varepsilon\,j(h))-X_t) - \mathcal{D} \Psi_t(F_\star(\bullet))\left[\langle D F_\star(\bullet), h\rangle_{\HHH}\right]\right\|_{\R^d}>\lambda\right]\nonumber\\
&= \P\left[\left\|\frac{1}{\varepsilon}(\Gamma_{h,t}(\varepsilon,\bullet)-\Gamma_{h,t}(0,\bullet)) - \mathcal{D} \Psi_t(F_\star(\bullet))\left[\langle D F_\star(\bullet), h\rangle_{\HHH}\right]\right\|_{\R^d}>\lambda\right].
 \end{align}
In view of \eqref{eq:DerivativeOfXt} and the fact that almost sure convergence implies convergence in probability, the right-hand side of \eqref{eq:ProbaCv} tends to zero as $\varepsilon\to 0$, for every $\lambda>0$. This shows that $X_t$ satisfies \textbf{(SGD)}, with
 \begin{align}\label{eq:DXt}
	\left<DX_t, h\right>_{\HHH}=\mathcal{D} \Psi_t(F_\star)\left[\langle D F_\star, h\rangle_{\HHH}\right].
 \end{align}
 Proposition \ref{pro:Nualart} then yields the expression \eqref{eq:MallDiffXt}.

\medskip

 We now turn to \textbf{(RAC)}. According to Theorem \ref{thm:ray}, for each $1 \leq \ell \leq m$ and $h\in \HHH$, fix $\beta \in (0,H-\frac{1}{2})$, set $1-\alpha=H-\beta$, and let $(\tilde{F}^{\,h,\ell}_t)_{t\in [0,T]}$ be the following version of $F^\ell$:
	\begin{enumerate}[(a)]
	\item For any $t\in [0,T]$,  $\tilde{F}^{\,h,\ell}_t=F^\ell_t$ $\P$-a.s.;
	\item For every $\omega\in\Omega$, the mapping
	$
	\R_+\times\R \ni (t,\varepsilon)\longmapsto \widetilde F^{\,h,\ell}_t\big(\omega+\varepsilon\,j(h)\big)
	$
	is continuous.
	\item For all $h\in \HHH$, and  all $\omega\in\Omega$, the mapping from $\R$ into $W_{2}^{1-\alpha}\left(0, T ; \mathbb{R}^{m}\right)$ defined by $\varepsilon\mapsto \tilde{F}^{h,\ell}_{\star}(\omega + \varepsilon j(h))$ is of class $\mathcal{C}^{\infty}$ on $\R$.
	\end{enumerate}
Set $\tilde{F}^h_t=(\tilde{F}^{\,h,1}_t, \ldots, \tilde{F}^{\,h,m}_t)$ and $\tilde{X}^h_{t}=\Psi_t(\tilde{F}_\star^h)$. Property (b) above guarantees that, for every $h\in\HHH$ and $\omega\in\Omega$, the map $s\mapsto \tilde{F}^{h}_{s}(\omega)$ is continuous. On the other hand, Hypothesis $\mathbf{(H_2)}$ and the Kolmogorov continuity theorem ensure that $s\mapsto F_{s}(\omega)$ is also continuous. Combining this with (a), we obtain, for every $h\in\HHH$ and almost every $\omega\in\Omega$,
\begin{equation}\label{eq:YEqualZ0}
	\tilde{F}^h_s(\omega)=F_{s}(\omega), \qquad \text{ for all } s\in[0,T].
\end{equation}
Consequently, for every $h\in\HHH$ and almost every $\omega\in\Omega$,
\begin{equation}\label{eq:XEqualPsiZ0}
	\tilde{X}^h_{t}(\omega)=\Psi_t(\tilde{F}^h_\star (\omega))= \Psi_t(F_\star(\omega))=X_{t}(\omega).
\end{equation}
Moreover, Proposition \ref{pro:Nualart} gives the differentiability of $\Psi_t$, while property (c) above ensures that, for every $h\in\HHH$ and $\omega\in\Omega$, the map $\R\ni \varepsilon\mapsto \tilde{F}^h_{\star}(\omega + \varepsilon j(h))\in W_{2}^{1-\alpha}\left(0, T ; \mathbb{R}^{m}\right)$ is of class $\mathcal{C}^{\infty}$ on $\R$.
Hence, for every $h\in\HHH$ and every $\omega\in\Omega$, the map $\R\ni \varepsilon\mapsto \tilde{X}^h_{t}(\omega + \varepsilon j(h))$ is absolutely continuous on $\R$.
This establishes \textbf{(RAC)} and thereby completes the proof of Theorem \ref{thm:MallDiff}.
 \end{proof}
 
 \begin{Rmk}
 In \cite{nualart2009malliavin}, the authors study the solution $(X_t)_{t \in [0,T]}$ of equation \eqref{eq:SDEsDrivenF} driven by a multidimensional fractional Brownian motion with Hurst parameter $H \in (1/2,1)$, and prove that $X_t$ belongs to the local space $(\mathbb{D}_{\text{loc}}^{1,2})^d$ for every $t \in [0,T]$. Theorem \ref{thm:MallDiff} sharpens this conclusion by showing that $X_t \in (\mathbb{D}^{1,\infty})^d$.
 \end{Rmk}

 \begin{Rmk}
In the Gaussian setting, one usually requires the components of the driving process to be independent, as reflected by the Kronecker symbol in the covariance function \eqref{eqn:covFBM}. No such assumption is needed in Theorem \ref{thm:MallDiff}, which therefore refines the known results even in the Gaussian case.
 \end{Rmk}

 Using the {\it sewing lemma} in Hilbert spaces \cite[Proposition 2.1]{li2022mild}, we can also provide an $\HHH$-valued representation of the Malliavin derivative of $X$.

\begin{Thm}\label{thm:RepresentationDXt}
Under the conditions of Theorem \ref{thm:MallDiff}, we have, for all $1 \leq k \leq d$ and $t\in (0,T]$
\begin{enumerate}[(a)]
	\item 	
	\begin{align}
		D X_t^k &=  \sum_{\ell=1}^{m} \int_{0}^{t} \Theta_{t}^{k l}(r) d DF^l(r)\nonumber\\
		&:=\lim_{|\mathcal{P}|\to 0}\sum_{\ell=1}^{m}\sum_{[u,v]\in \mathcal{P}} \Theta_{t}^{k l}(u) (DF^l(v)-DF^l(u)),
	\end{align}
	where $\Theta^{k,\ell}_t$ is given by \eqref{eq:Theta} and the above limit is in $\HHH$, along an arbitrary sequence of partitions of $[0,t]$ with mesh tending to zero.
	Furthermore, for any $\beta\in (0, H-\frac{1}{2})$ there exists a constant $C>0$ such that
	\begin{align}\label{eq:estimate}
		&\left\|\sum_{\ell=1}^{m} \int_{s}^{t} \Theta_{t}^{k l}(r) d DF^l(r)- \sum_{\ell=1}^{d}\Theta_{t}^{k l}(s)\big(DF^l(t)-D^{(l)}F(s)\big)\right\|_{\HHH}\\
		&\leq C \sum_{\ell=1}^{m}\|\Theta_{t}^{k l}\|_{\mathcal{C}^{H-\beta_1}} \|DF^l\|_{\mathcal{C}^{H-\beta_2}} |t-s|^{2(H-\beta)};\nonumber
	\end{align}	
	uniform in $s,t\in [0,T]$. 
	\item for all $h \in \mathfrak{H}$, 
	\begin{align}
		&\left<D X_t^k, h\right>_{\HHH} =  \sum_{\ell=1}^{m} \int_{0}^{t} \Theta_{t}^{k l}(s) d \left<DF^l(s), h\right>_{\HHH}\nonumber\\
		&:=\lim_{|\mathcal{P}|\to 0}\sum_{\ell=1}^{m}\sum_{[u,v]\in \mathcal{P}} \Theta_{t}^{k l}(u) \left<DF^l(v)-DF^l(u),h\right>_{\HHH},	
	\end{align}
	where the limit is along an arbitrary sequence of partitions of $[0,t]$ with mesh tending to zero.
\end{enumerate}
\end{Thm}

\begin{proof}
	We only prove point (a), as point (b) is obtained along the same lines from the {\it sewing lemma}. Fix $1 \leq k \leq d$. Hypothesis $\mathbf{(H_2)}$ together with the Kolmogorov continuity theorem imply that, for every $1 \leq \ell \leq m$, the map $s\mapsto DF^\ell_s$ belongs almost surely to $C^{H-\beta}([0,T], \HHH)$.
	Moreover, \cite[Proposition 9]{nualart2009malliavin} ensures that $s\mapsto \Theta_{t}^{k,\ell}(s)$ belongs to $C^{H-\beta}([0,T], \R)$ for every $1 \leq k \leq d$, $1 \leq \ell \leq m$ and $t\in (0,T]$.
	Invoking \cite[Proposition 2.1]{li2022mild}, we conclude that, for every $1 \leq k \leq d$, $1 \leq \ell \leq m$ and $t\in (0,T]$, the Young integral
	\begin{align}\label{eq:YoungIntegral}
		\int_{0}^{t} \Theta_{t}^{k,\ell}(s) d DF^l(s):=\lim_{|\mathcal{P}|\to 0}\sum_{[u,v]\in \mathcal{P}} \Theta_{t}^{k l}(u) (DF^l(v)-DF^l(u))
	\end{align}
	is well defined as a limit in $\HHH$ along any sequence of partitions of $[0,t]$ whose mesh tends to zero.
	Estimate \eqref{eq:estimate} is then a consequence of \eqref{eq:YoungIntegral} and the bound (2.9) of \cite[Proposition 2.1]{li2022mild}.
	%Now, for the second point,  we have 
	%\begin{align*}
	%	\left<\int_{0}^{t} \Theta_{t}^{k l}(s) d D^{(l)}F(s),h\right>_{\HHH} &= \lim_{|\mathcal{P}|\to 0}\sum_{[u,v]\in \mathcal{P}} \Theta_{t}^{k l}(u) \left<D^{(l)}F(v)-D^{(l)}F(u),h\right>_{\HHH}\\
	%	&= \int_{0}^{t} \Theta_{t}^{k l}(s) d \left<D^{(l)}F(s),h\right>_{\HHH}
	%\end{align*}
\end{proof}

\section{Absolute continuity of the law of the solution}\label{sec:density}

In this section, we show that the law of the solution to the stochastic differential equation \eqref{eq:SDEsDrivenF} is absolutely continuous with respect to the Lebesgue measure. Our main tool is the Bouleau--Hirsch criterion, a cornerstone of Malliavin calculus \cite{Bouleau1986}.

\begin{Thm} \label{thm:BouleauHirsch}
	Let $d$ be a positive integer and let $\mathbf{Y}=\left(Y_1, \ldots, Y_d\right)$ be a random vector satisfying the following two conditions:
\begin{enumerate}[(i)]
	\item $Y_i \in \mathbb{D}^{1, p}$ for all $i=1, \ldots, d$, for some $p \in(1, \infty)$;
	\item $\operatorname{det}\left(\Gamma_{\mathbf{Y}}\right)>0$ $\mathbb{P}$-a.s., where $\Gamma_{\mathbf{Y}}=\left(\left\langle D Y_i, D Y_j\right\rangle_{\mathfrak{H}}\right)_{1 \leq i, j \leq d}$ denotes the Malliavin matrix of $\mathbf{Y}$.
\end{enumerate}
Then the law of $\mathbf{Y}$ is absolutely continuous with respect to the Lebesgue measure on $\mathbb{R}^d$.
\end{Thm}

The absolute continuity of the law of the solution to a stochastic differential equation is typically obtained under various non-degeneracy conditions, imposed both on the coefficients and on the driving process. In what follows, we take as a starting point the formulation of these conditions proposed in \cite{MR2680405} within the framework of rough path theory.

We first impose the following ellipticity assumption on the coefficients $(b_k)_{1 \leq k \leq d}$ and $(\sigma_{k,\ell})_{1 \leq k \leq d, 1 \leq \ell \leq m}$ of \eqref{eq:SDEsDrivenF}.
\begin{center}
$\mathbf{(H_{3})}$ For any $x\in \R^d$, the vector space spanned by $\{\sigma_{1,\ell}(x),\ldots,\sigma_{d,\ell}(x)\,,\; 1\leq \ell \leq m\}$ is equal
	to $\R^d$.
\end{center}

The conditions imposed on the driving process in \cite{MR2680405} are specific to the Gaussian rough path setting. We first recall them before proposing their counterparts in our non-Gaussian framework.
\begin{center}
$\mathbf{(GRPH_{4})}$ The components of the driving Gaussian process are independent.
\end{center}
\begin{center}
$\mathbf{(GRPH_{5})}$ For any $g=(g^1,\dots,g^d) \, : \, [0,T] \to \R^d$ with bounded $p$-variation\footnote{In \cite{MR2680405}, the driving process is assumed to be $p$-rough path.},
\[\left( \int_0^T g_s dh_s = \sum_{k=1}^d \int_0^T g_s^k dh_s^k =0 ~\forall \, h \in j(\mathfrak{H})\right)\Rightarrow g=0. \]
\end{center}
In our setting, we use conditions analogous to $\mathbf{(GRPH_{4})}$ and $\mathbf{(GRPH_{5})}$, but formulated at the level of Malliavin derivatives. For the independence assumption, we adopt the following hypothesis.
\begin{center}
$\mathbf{(H_{4})}$ The components of the driving process in \eqref{eq:SDEsDrivenF} are Malliavin independent: for all $1 \leq \ell \neq \ell' \leq m$ and $t \in [0,T]$, almost surely,
\[ \langle DF_t^{\ell} , DF_t^{\ell'} \rangle_\mathfrak{H}=0.\]
\end{center}
When $F$ is Gaussian, it is straightforward to check that $\mathbf{(GRPH_{4})}$ and $\mathbf{(H_{4})}$ are equivalent. As for $\mathbf{(GRPH_{5})}$, the bounded $p$-variation assumption is natural in the rough path setting; in our context, it is more natural to work with a Hölder-type condition.
\begin{center}
$\mathbf{(H_{5})}$ The driving process in \eqref{eq:SDEsDrivenF} is non-degenerate, if for any fixed $t \in (0,T]$, and any $\R^m$ valued stochastic process $(Y_r)_{r \in [0,t]}$ such that the sample paths of $Y$ are $\gamma$-H\"{o}lder continuous, for some $\gamma>\frac{1}{2}$, we have
\begin{align*}
	&\bigg(\exists \Omega_0\in \mathcal{F}, \text{ with } \P(\Omega_0)>0, \text{ s.t., } \forall\omega\in \Omega_0, \text{ and } 1 \leq \ell \leq m, \\
	&\qquad\qquad\qquad\qquad\left\|\int_{0}^{t} Y^\ell_s(\omega)dDF^{\ell}(s,\omega)\right\|_{\HHH}=0\bigg)\\
	&\Longrightarrow  \bigg(\exists \Omega_1\in \mathcal{F}, \text{ with } \P(\Omega_1)>0, \; \exists s\in [0,t], \text{ s.t., }\forall\omega\in \Omega_1 \,  Y_s(\omega)=0 \;\;\; \bigg).
\end{align*}
\end{center}

We are now ready to state and prove the main result of this section.

\begin{Thm}\label{thm:AbsoluteContinuity}
	Let $F=\{F_t=(I_q(f_t^1),\ldots,I_q(f_t^m))\,,\;t\in \R_+\}$ be an $\R^m$-valued chaos process on an abstract Wiener space $(\Omega,\mathcal{F},\P,\mathfrak{H})$, satisfying hypotheses $\mathbf{(H_{2})}$, $\mathbf{(H_{4})}$ and $\mathbf{(H_{5})}$. Assume moreover that the coefficients $(b_k)_{1 \leq k \leq d}$ and $(\sigma_{k,\ell})_{1 \leq k \leq d,\, 1 \leq \ell \leq m}$ of the stochastic differential equation \eqref{eq:SDEsDrivenF} satisfy hypotheses $\mathbf{(H_1)}$ and $\mathbf{(H_3)}$.
Then, for every $t\in(0,T]$, the law of $X_t$, where $(X_t)_{t \in [0,T]}$ is the unique solution to \eqref{eq:SDEsDrivenF}, is absolutely continuous with respect to the Lebesgue measure on $\R^d$.
\end{Thm}

\begin{proof}
Parts of the proof are inspired by that of \cite[Theorem 8]{nualart2009malliavin}. Fix $t\in (0,T]$. By Theorem \ref{thm:MallDiff}, $X^k_t \in \mathbb{D}^{1,\infty}$ for every $1 \leq k \leq d$. In view of the Bouleau--Hirsch criterion (Theorem \ref{thm:BouleauHirsch}), it thus suffices to show that the determinant of the Malliavin matrix $\Gamma_{X_t}$ of $X_t$ is almost surely positive.
Let $\{e_n\,,\; n\geq 1\}$ be an orthonormal basis of $\HHH$. We can write
\begin{align*}
	\Gamma_{X_t}= (M_{k,k'})_{1 \leq k,k'\leq d}, 
\end{align*}
where
\begin{align*}
	M_{k,k'} = \left<D X^k_t, D X^{k'}_t\right>_{\HHH}= \sum_{n=1}^{\infty} \left<D X^k_t, e_n\right>_{\HHH}\left<D X^{k'}_t, e_n\right>_{\HHH}. 
\end{align*}
Assume, for contradiction, that $\det\Gamma_{X_t}=0$ on some $\Omega_0\in \mathcal{F}$ with $\P(\Omega_0)>0$. Then there exists a nonzero $\R^d$-valued random vector $a$ such that $a^T \Gamma_{X_t} a=0$ on $\Omega_0$. On the other hand, we may rewrite $a^T \Gamma_{X_t} a$ as
\begin{align}\label{eq:AtGammaAt}
	a^T \Gamma_{X_t} a &= \sum_{n=1}^{\infty} \left|\left< \left<D X_t, e_n\right>_{\HHH},a\right>_{\R^d}\right|^2,
\end{align}
where $\left<\cdot,\cdot\right>_{\R^d}$ is the Euclidean inner product on $\R^d$ and
$$\left<D X_t, e_n\right>_{\HHH} = \left(\left<D X^1_t, e_n\right>_{\HHH}, \ldots, \left<D X^d_t, e_n\right>_{\HHH}\right).$$
By \eqref{eq:DXt}, for almost every $\omega\in\Omega$, every $n\geq 1$ and every $1 \leq k \leq d$,
\begin{align*}
	\left<D X_t^k(\omega), e_n\right>_{\HHH} =  \mathcal{D} \Psi_t^k(F_\star(\omega))\left[\langle D F_\star(\omega), e_n\rangle_{\HHH}\right].
\end{align*}
Moreover, using \eqref{eqn:Nualartbis:implict}, we have, for almost every $\omega\in\Omega$, every $n\geq 1$ and every $1 \leq k \leq d$,
\begin{align*}
	&\mathcal{D}  \Psi_t^k(F_\star(\omega))\left[\langle D F_\star(\omega), e_n\rangle_{\HHH}\right] \\
	&= \left(\left(-\mathcal{D}_2T(0,X(\omega))^{-1}\circ \mathcal{D}_1T(0,X(\omega))\right)\left[\langle D F_\star(\omega), e_n\rangle_{\HHH}\right]\right)^k_t,
\end{align*}
where $T$ is the operator defined in \eqref{eqn:Nualartbis}. Combining this with \eqref{eq:AtGammaAt}, we deduce that, for every $\omega\in\Omega_0$ and every $n\geq 1$,
\begin{align*}
	\left< \left(\mathcal{D}_2T(0,X(\omega))^{-1}\circ \mathcal{D}_1T(0,X(\omega))\right)\left[\langle D F_\star(\omega), e_n\rangle_{\HHH}\right]_t , a(\omega) \right>_{\R^d} =0. 
\end{align*}  
Since $\mathcal{D}_2T(0,X)^{-1}$ is a linear homeomorphism (see Proposition \ref{pro:Nualartbis}), expression \eqref{eqn:Nualartbis:frechet1} then yields, for every $\omega\in\Omega_0$ and $n\geq 1$,
\begin{align*}
	0&=\left< \mathcal{D}_1T(0,X(\omega))\left[\langle D F_\star(\omega), e_n\rangle_{\HHH}\right]_t , a(\omega) \right>_{\R^d}\\
	&=  \sum_{k=1}^{d}\left(\sum_{\ell=1}^{m}\int_{0}^{t}  \sigma_{k,\ell}(X_s(\omega)) d \left<DF^{l}(s,\omega), e_n\right>_{\HHH}\right)a^k(\omega)\\
	&= \left<\sum_{k=1}^{d}\sum_{\ell=1}^{m}\int_{0}^{t}  a^k(\omega)\sigma_{k,\ell}(X_s(\omega)) dDF^{l}(s,\omega), e_n\right>_{\HHH}.
\end{align*}
Invoking the Malliavin independence of the components of $F$, i.e.\ hypothesis $\mathbf{(H_4)}$, we obtain, for every $\omega\in\Omega_0$,
\begin{align*}
	0&= \left\|\sum_{k=1}^{d}\sum_{\ell=1}^{m}\int_{0}^{t}  a^k(\omega)\sigma_{k,\ell}(X_s(\omega)) dDF^{\ell}(s,\omega)\right\|_{\HHH}^2\\
	&= \sum_{\ell=1}^{m}\left\|\sum_{k=1}^{d}\int_{0}^{t}  a^k(\omega)\sigma_{k,\ell}(X_s(\omega)) dDF^{\ell}(s,\omega)\right\|_{\HHH}^2,
\end{align*}
from which we conclude that, for every $1 \leq \ell \leq m$ and every $\omega\in\Omega_0$,
\begin{align}\label{eq:ZeroNorm}
	\left\|\sum_{k=1}^{d}\int_{0}^{t}  a^k(\omega)\sigma_{k,\ell}(X_s(\omega)) dDF^{\ell}(s,\omega)\right\|_{\HHH}^2=0.
\end{align}
Applying hypothesis $\mathbf{(H_5)}$, we obtain the existence of $\Omega_1\in \mathcal{F}$ with $\P(\Omega_1)>0$, and of $s\in [0,T]$, such that, for every $1 \leq \ell \leq m$ and every $\omega\in \Omega_1$,
\begin{align*}
	\sum_{k=1}^{d} a^k(\omega)\sigma_{k,\ell}(X_s(\omega))=0,
\end{align*}
which contradicts hypothesis $\mathbf{(H_3)}$. Consequently, $\det \Gamma_{X_t}>0$ almost surely, and this completes the proof.
\end{proof}

\section{Application: the case of multidimensional Hermite processes}\label{sec:Hermite}

Throughout this section, $(\Omega,\mathcal{F}, \mathbb{P}, \mathfrak{H})$ denotes the classical Wiener space: $\Omega= C_0(\R,\mathbb{R}^m)$ is the space of continuous functions $\omega: \R\to \mathbb{R}^m$ with $\omega(0)=0$, equipped with the Fréchet metric; $\mathbb{P}$ is the two-sided Wiener measure, i.e.\ the unique probability measure on $\Omega$ under which the canonical process
\begin{equation}\label{eqn:multibrownian}
B=\{B_t=(B^{1}_t,\ldots,B^{m}_t),\;  t\in  \R\}
\end{equation}
is a two-sided $m$-dimensional Brownian motion; $\mathcal{F}$ is the completion of the Borel $\sigma$-algebra of $\Omega$ with respect to $\mathbb{P}$; $\mathfrak{H}=L^2(\R,\mathbb{R}^m)$; and, for all $h \in \mathfrak{H}$,
\[ j(h) \, : \,  \R\to \mathbb{R}^m \, : \, t \mapsto \left(\int_0^t h^1(s) \, ds, \dots, \int_0^t h^m(s) \, ds \right). \]
In this setting, the isonormal Gaussian process $\{X_h \, : \, h \in \mathfrak{H}\}$ is given, for every $h \in \mathfrak{H}$, by the Wiener integral
\begin{equation}\label{eqn:isodanslecashermimult}
 X_h = \sum_{\ell=1}^m \int_{\R} h^\ell(t) dB^\ell_t.
\end{equation}
In this context, the Malliavin derivative of the cylindrical random variable $F$ in \eqref{intro:defmallia} takes the form of the $m$-dimensional stochastic process $DF=\{D_rF=(D^{(1)}_rF,\ldots,D^{(m)}_rF)\,,\; r\in\R\}$, where, for every $1 \leq \ell \leq m$ and $r \in \R$,
  \[D^{(\ell)}_rF=\sum_{i=1}^{n}\frac{\partial f}{\partial x_i} (X_{h_1},\ldots,X_{h_n})h_i^\ell(r).\]
This formalization allows us to write, for every $F,G \in \mathbb{D}^{1,p}$ ($p \geq 1$),
  \[\left<DF\,,\,DG\right>_{\HHH}=\sum_{\ell=1}^{m}\int_{\mathbb{R}}D^{(\ell)}_{r}F\, D^{(\ell)}_{r}G\,dr. \]

Throughout this section, $I_q$ denotes the $q$th stochastic integral (as defined in Subsection \ref{subsec:abstractwienerspaces}) with respect to the isonormal Gaussian process \eqref{eqn:isodanslecashermimult}, i.e.\ the $q$th multiple Wiener integral with respect to the $m$-dimensional Brownian motion \eqref{eqn:multibrownian}.

Let us first consider the case $m=1$ and introduce the one-dimensional Hermite process.
\begin{Def}\label{Def:Hermite}
The one-dimensional Hermite process $Z^{H,q}=\left(Z_t^{H,q}\right)_{t \in \R_+}$ of order $q\geq 1$ and self-similarity parameter $H \in\left(\frac{1}{2}, 1\right)$ is defined by
\begin{align}\label{eq:Zt}
	Z^{H,q}_t=I_q(L_t^{H, q}) \qquad \text{where} \qquad L_t^{H, q}(\xi)=c(H, q) \int_0^t \prod_{j=1}^q\left(s-\xi_j\right)_{+}^{H_0-\frac{3}{2}} ds,
\end{align}
$$
H_0=1+\frac{H-1}{q} \in\left(1-\frac{1}{2 q}, 1\right), \quad\text{and}\quad c(H, q)=\sqrt{\frac{H(2 H-1)}{q!\,\beta^q\left(H_0-\frac{1}{2}, 2-2 H_0\right)}}.
$$
Here $\beta$ stands for the Beta function and we use the convention, for each $(\theta, \alpha) \in \R^2$,
		\[\theta^{\alpha}_+:=\left\{ \begin{array}{rl}\theta^{\alpha}, & \text{if }\theta>0,\\ 0, &  \text{otherwise}. \end{array}\right. \]
\end{Def}

The family of Hermite processes provides paradigmatic examples of chaotic random processes. It includes the Rosenblatt process, obtained for $q=2$, and the fractional Brownian motion, obtained for $q=1$; the latter is the only Gaussian element of this family. A direct computation (see e.g.\ \cite[Proposition 2.1]{tudorbook}) shows that, for every $s,t \geq 0$,
\[ \langle L_t^{H, q}, L_s^{H, q} \rangle_{L^2(\R^q,\R)}=\frac{q!}{2}\left(t^{2H}+s^{2H}-|t-s|^{2H}\right),\]
which yields, by the polarization formula,
\begin{equation}\label{eqn:boundL2kernel}
\|L_t^{H, q}-L_s^{H, q} \|_{L^2(\R^q,\R)} \leq q!\,|t-s|^{2H}.
\end{equation}
In particular, hypothesis $\mathbf{(H_2)}$ is satisfied by the one-dimensional Hermite process $Z^{H,q}$. We refer the reader to Tudor's monograph \cite{tudorbook} for a detailed presentation of this class of processes and of their properties, which include stationarity of increments and $H$-self-similarity. For our purposes, the key feature is that $(Z_t^{H,q})_{t\in\R_+}$ is adapted to the filtration defined, for $t\in\R$, by
  $$\mathcal{F}_t=\sigma\{B(\mathbbm{1}_{A})\,: A \text{ is a Borel } \lambda\text{-finite subset of } (-\infty,t]\}\vee \mathcal{N},$$
where $\mathcal{N}$ denotes the collection of $\mathbb{P}$-null subsets of $\Omega$ and $\lambda$ is the Lebesgue measure on $\R$. In view of \cite[Corollary 1.2.1]{Nualart}, we therefore have, for every $t\in\R_+$,
  \begin{equation}\label{eq:adaptfromnualart}
  D_r Z_t^{H,q} = 0 \qquad \text{for } \P\otimes\lambda\text{-a.e.\ } (\omega,r)\in \Omega\times (t,\infty).
  \end{equation}

We now turn to the definition of the multidimensional Hermite process.

\begin{Def}\label{Def:Hermite288}
	Let $m\in \N^*$, $q\geq 1$ and $(Z^{H,q,\ell} :=\{Z^{H,q,\ell}_{t}\}_{t\in \R_+})_{1 \leq \ell \leq m}$ be $m$ independent Hermite processes of order $q$ and self-similarity parameter $H$. The $\R^m$-valued stochastic process $Z^{H,q}=(Z^{H,q,1},\ldots,Z^{H,q,m})$ is called the $m$-dimensional Hermite process of order $q$ and self-similarity parameter $H$.
\end{Def}

Fix $m\in \N^*$ and $q\geq 1$, and denote by $I_q^{B^\ell}$ the $q$th multiple integral with respect to the Brownian motion $B^\ell$ (that is, the $\ell$th component in \eqref{eqn:multibrownian}), for $1 \leq \ell \leq m$. Setting
\begin{align}\label{eq:Zt1}
	Z^{H,q,\ell} =\{I^{B^\ell}_q(L_t^{H, q})\}_{t \geq 0}, \qquad \text{with } L_t^{H, q} \text{ given by } \eqref{eq:Zt},
\end{align}
the process $Z^{H,q}:=(Z^{H,q,1},\ldots,Z^{H,q,m})$ is an $m$-dimensional Hermite process of order $q$ and self-similarity parameter $H$. However, its components are built from multiple integrals with respect to \emph{distinct} isonormal Gaussian processes, so $Z^{H,q}$ does not directly fit into the framework of Definition \ref{def:drivingprocess}. Our first task in this section is therefore to provide an explicit construction realizing the components of $Z^{H,q}$ as multiple integrals with respect to a \emph{single} isonormal Gaussian process, namely \eqref{eqn:isodanslecashermimult}.

Recall that, with $\mathfrak{H}=L^2(\R,\mathbb{R}^m)$, one has for any $q \geq 1$
	\begin{align*}
		\HHH^{\otimes q} = L^2(\R^q, \R^{m^q}) \qquad \text{and} \qquad \HHH^{\odot  q} = L^2_s(\R^q, \R^{m^q}).
	\end{align*}
Given $h_q\in L^2_s(\R^q, \R)$ and $1 \leq \ell \leq m$, let $\hat{h}_{\ell,q}=(\hat{h}_{\ell,q}^{i_1,\ldots, i_q})_{1\leq i_1,\ldots, i_q \leq m}$ denote the element of $L^2_s(\R^q, \R^{m^q})$ defined by
\begin{align}\label{eq:hath299}
	\hat{h}_{j,q}^{i_1,\ldots, i_q} =
\begin{cases}
	\hat{h}_{j,q}^{i_1,\ldots, i_q}= h_q,       & \text{ if } i_1=\ldots=i_q=\ell,\\
	\hat{h}_{j,q}^{i_1,\ldots, i_q}=0,    & \text{ otherwise}.
\end{cases}
\end{align}
A direct verification shows that, if $g_1,\dots,g_q\in L^2(\R,\R)$ and
\begin{equation}\label{eqn:simplefunctionh}
h_q=g_1\odot \dots \odot g_q \in L^2_s(\R^q, \R),
\end{equation}
then, setting for any $1 \leq k \leq q$ and $1 \leq \ell \leq m$,
\begin{equation}\label{def:defofgtilde}
 \tilde{g}_{k,\ell} = (\underbrace{0,\;\ldots\;,0}_{\substack{(\ell-1)\\ \text{elements}}},\; g_k\;, \underbrace{0,\;\cdots\;,0}_{\substack{(m-\ell)\\ \text{elements}}}),
\end{equation}
one has $\tilde{g}_{1,\ell} \odot \dots \odot \tilde{g}_{q,\ell} =\hat{h}_{\ell,q}$. This identity is the cornerstone of the following proposition.

\begin{Prop}\label{Prop:HermiteMulti}
	Let $h_q\in L^2_s(\R^q, \R)$ and, for $1 \leq \ell \leq m$, let $\hat{h}_{\ell,q}\in L^2_s(\R^q, \R^{m^q})$ be as defined in \eqref{eq:hath299}. Then
	\begin{align}\label{eq:IqIq}
		I_q(\hat{h}_{\ell,q})= I_q^{B^\ell}(h_q).
	\end{align}
\end{Prop}
\begin{proof}
	We first treat the case where $h_q$ has the simple form \eqref{eqn:simplefunctionh}, and argue by induction on $q$. For $q=1$, the identity $I_1(\hat{h}_{\ell,1})=I_1^{B^\ell}(h_1)$ is immediate for every $1 \leq \ell \leq m$. Assume now that \eqref{eq:IqIq} holds at order $q$, and let us establish it at order $q+1$. Using \eqref{eqn:reintegration}, we obtain, for any $1 \leq \ell \leq m$,
	\begin{align*}
		I_{q+1}(\hat{h}_{\ell,q+1}) =& I_{q+1}(\tilde{g}_{1,\ell} \odot \dots \odot \tilde{g}_{q+1,\ell} ) \\
		&= \frac{1}{q+1}  \sum_{k=1}^{q+1} \delta\left(I_q(\tilde{g}_{1,\ell} \odot \dots \odot \widehat{\tilde{g}_{k,\ell}} \odot \dots \odot \tilde{g}_{q+1,\ell}) \tilde{g}_{k,\ell}\right),
\end{align*}
where the hat over $\tilde{g}_{k,\ell}$ indicates that this factor is omitted from the symmetric tensor product. For each $1 \leq k \leq q+1$, successive applications of \eqref{integrationbypart} and \eqref{eqn:derivativeofintegral} yield
\begin{align*}
\delta &\left(I_q(\tilde{g}_{1,\ell} \odot \dots \odot \widehat{\tilde{g}_{k,\ell}} \odot \dots \odot \tilde{g}_{q+1,\ell}) \tilde{g}_{k,\ell}\right) \\
&=I_q(\tilde{g}_{1,\ell} \odot \dots \odot \widehat{\tilde{g}_{k,\ell}} \odot \dots \odot \tilde{g}_{q+1,\ell}) \delta(\tilde{g}_{k,\ell}) \\ & \qquad - \left\langle D I_q(\tilde{g}_{1,\ell} \odot \dots \odot \widehat{\tilde{g}_{k,\ell}} \odot \dots \odot \tilde{g}_{q+1,\ell}),\tilde{g}_{k,\ell} \right \rangle_\mathfrak{H} \\
&= I_q(\tilde{g}_{1,\ell} \odot \dots \odot \widehat{\tilde{g}_{k,\ell}} \odot \dots \odot \tilde{g}_{q+1,\ell}) I_1(\tilde{g}_{k,\ell})\\
& \qquad - \sum_{\substack{n=1 \\ n \neq k}}^{q+1} I_{q-1}(\tilde{g}_{1,\ell} \odot \dots \odot \widehat{\tilde{g}_{n,\ell}} \odot \dots \odot \widehat{\tilde{g}_{k,\ell}} \odot \dots \odot \tilde{g}_{q+1,\ell})\left \langle \tilde{g}_{n,\ell}  ,\tilde{g}_{k,\ell} \right \rangle_\mathfrak{H}
\end{align*}
Combining the induction hypothesis with the definition \eqref{def:defofgtilde}, we get
\begin{align*}
	I_{q+1}(\hat{h}_{\ell,q+1}) &= \frac{1}{q+1}  \sum_{k=1}^{q+1} \left( I_q^{ B^\ell} (\tilde{g}_{1} \odot \dots \odot \widehat{g_k} \odot \dots \odot g_{q+1}) I_1^{B^\ell}(g_k) \right. \\
	& - \sum_{\substack{n=1 \\ n \neq k}}^{q+1} \left.  I_{q-1}^{B^\ell}(g_1 \odot \dots \odot \widehat{g_n} \odot \dots \odot \widehat{g_k} \odot \dots \odot g_{q+1})\left \langle g_n  ,g_k \right \rangle_{L^2(\R,\R)} \right) .
	\end{align*}
	Running the same argument backwards, we deduce the claimed identity
	\[ I_{q+1}(\hat{h}_{\ell,q+1})=I_{q+1}^{B^\ell}(h_q).\]
The general case then follows by a density argument.
\end{proof}

As a consequence of Proposition \ref{Prop:HermiteMulti}, we can now give a definition of the $m$-dimensional Hermite process $(Z^{H,q,1},\ldots,Z^{H,q,m})$ that relies on stochastic integration with respect to a single isonormal Gaussian process. Indeed, it suffices to set, for any $t \geq 0$ and $1 \leq \ell \leq m$,
\begin{align}\label{eqn:defofmultihermitekernel}
		Z^{H,q,\ell}_t = I_q(\hat{L}_{j,t}^{H,q}),\
		\end{align} 
		where
		\[	
		 \hat{L}_{\ell,t}^{H,q} =
\begin{cases}
	(\hat{L}_{\ell,t}^{H,q})^{i_1,\ldots, i_q}= L_t^{H,q},       & \text{ if } i_1=\ldots=i_q=\ell\\
	(\hat{L}_{\ell,t}^{H,q})^{i_1,\ldots, i_q}=0,    & \text{ otherwise}
\end{cases}
	\]
with $L_t^{H,q}$ as in equation \eqref{eq:Zt}.

We now verify that, with this definition, the $m$-dimensional Hermite process $(Z^{H,q,1},\ldots,Z^{H,q,m})$ fits the framework of the present paper: it satisfies hypotheses $\mathbf{(H_2)}$, $\mathbf{(H_4)}$ and $\mathbf{(H_5)}$, so that Theorems \ref{thm:MallDiff}, \ref{thm:RepresentationDXt} and \ref{thm:AbsoluteContinuity} apply to stochastic differential equations driven by multidimensional Hermite processes.

We first observe, combining \eqref{eqn:defofmultihermitekernel} with \eqref{eqn:boundL2kernel}, that for any $1 \leq \ell \leq m$ and $s,t \geq 0$,
 \[\|\hat{L}_{\ell,t}^{H,q}-\hat{L}_{\ell,t}^{H,q} \|_{\HHH^{\otimes q}}= \|L_t^{H, q}-L_s^{H, q} \|_{L^2(\R^q,\R)} \leq q!\,|t-s|^{2H}. \]
Hence hypothesis $\mathbf{(H_2)}$ holds for the $m$-dimensional Hermite process $$(Z^{H,q,1},\ldots,Z^{H,q,m})$$ defined in \eqref{eqn:defofmultihermitekernel}. That $\mathbf{(H_4)}$ is also satisfied follows from the next general proposition.

\begin{Prop}\label{prop:MalIndepHermite}
Let $h_q\in L^2_s(\R^q, \R)$ and, for $1 \leq \ell \leq m$, let $\hat{h}_{\ell,q}\in L^2_s(\R^q, \R^{m^q})$ be as defined in \eqref{eq:hath299}. Then the components of the chaotic process $$(I_q(\hat{h}_{1,q}),\dots,I_q(\hat{h}_{m,q}))$$ are Malliavin independent.
\end{Prop}

\begin{proof}
We first assume that $h_q$ has the simple form \eqref{eqn:simplefunctionh}. For any $1 \leq \ell \leq m$, applying \eqref{eqn:derivativeofintegral} yields
\[ 	D I_{q}(\hat{h}_{\ell,q}) = D I_{q}(\tilde{g}_{1,\ell} \odot \dots \odot \tilde{g}_{q,\ell}) = \sum_{k=1}^{q} I_{q-1}(\tilde{g}_{1,\ell} \odot \dots \odot \widehat{\tilde{g}_{k,\ell}} \odot \dots \odot \tilde{g}_{q,\ell}) \tilde{g}_{k,\ell},\]
with $\tilde{g}_{k,\ell}$ given by \eqref{def:defofgtilde}. From this definition, it follows immediately that for any $1 \leq \ell\neq \ell' \leq m$,
\[	\left<D I_{q}(\hat{h}_{\ell,q}),D I_{q}(\hat{h}_{\ell',q})\right>_{\HHH} = 0. \]
The general case is then obtained by density.
\end{proof}

It remains to show that the $m$-dimensional Hermite process defined in \eqref{eqn:defofmultihermitekernel} also satisfies hypothesis $\mathbf{(H_5)}$. To this end, we rely on the following lemma.

\begin{Lemma}\label{Prop:Zt9999}
	Let $Z^{H,q}=\{Z^{H,q}_t \}_{t\in \R_+}$ be the one-dimensional Hermite process of order $q\geq 1$ and self-similarity parameter $H\in(\frac{1}{2},1)$, as in Definition \ref{Def:Hermite}. Then, for every $t>0$ and $\varepsilon\in (0,t)$, the following equality in distribution holds:
		\begin{align}
			\int_{t-\varepsilon}^{t}|D_rZ^{H,q}_{t}-D_rZ^{H,q}_{t-\varepsilon}|^2\,dr \overset{\text{law}}{=} \varepsilon^{2H}\int_{0}^{1}|D_rZ^{H,q}_{1}|^2\,dr.
		\end{align}
	\end{Lemma}
	\begin{proof}
	Fix $t>0$ and $\varepsilon\in(0,t)$, and set $h=t-\varepsilon$. We have
	\begin{align}\label{eq:Iq1}
		\int_{t-\varepsilon}^{t} dx|D_{x}Z^{H,q}_{t}-D_{x}Z^{H,q}_{t-\varepsilon}|^2 &= 	\int_{h}^{\varepsilon +h} d{x}|D_{x}Z^{H,q}_{h + \varepsilon }-D_{x}Z^{H,q}_{h}|^2 \nonumber\\ 
		& = q^2	\,  \int_{h}^{\varepsilon +h} d{x} \left|I_{q-1}\left(L_{h,h+\varepsilon}^{H,q}({x} , \star) \right)\right|^2\nonumber\\
		& =q^2	\,  \int_{0}^{\varepsilon } dr \left|I_{q-1}\left(L_{h,h+\varepsilon}^{H,q}(r + h, \star) \right)\right|^2,
	\end{align}
	where the last equality follows from the change of variable $r=x - h$. Moreover, $L_{h,h+\varepsilon}^{H,q}(r+h,\star)$ can be written as
	\begin{align}\label{eq:Lhhepsilon}
		L_{h,h+ \varepsilon}^{H,q}(r + h, \star) &= c(H, q)\,\int_{h}^{h + \varepsilon}d s\left(\left(s-r-h\right)_{+}^{\frac{H-1}{q}-\frac{1}{2}}\prod_{j=1}^{q-1}\left(s-\star_j\right)_{+}^{\frac{H-1}{q}-\frac{1}{2}}\right)\nonumber\\
		&= c(H, q)\, \int_{0}^{\varepsilon}d \theta\left(\left(\theta - r\right)_{+}^{\frac{H-1}{q}-\frac{1}{2}}\prod_{j=1}^{q-1}\left(\theta + h -\star_j\right)_{+}^{\frac{H-1}{q}-\frac{1}{2}}\right),
	\end{align}
	where the last equality was obtained by the change of variable $\theta=s - h$, and $c(H, q)$ is the constant defined in Definition \ref{Def:Hermite}. Combining \eqref{eq:Iq1} and \eqref{eq:Lhhepsilon}, the right-hand side of \eqref{eq:Iq1} becomes
	\begin{align}\label{eq:ZZZt}
		%\int_{t-\varepsilon}^{t} dr|D_rZ^{H,q}_{t}-D_rZ^{H,q}_{t-\varepsilon}|^2 &= 	\int_{t-\varepsilon}^{t} dr|D_rZ^{H,q}_{(t-\varepsilon) + \varepsilon}-D_rZ^{H,q}_{t-\varepsilon}|^2 \\
		%& = q^2	\,  \int_{t-\varepsilon}^{t} dr \left|I_{q-1}\left(L_{t-\varepsilon,t}^{H,q}(r , \star) \right)\right|^2\\
		& q^2	c^2(H,q)\,\int_{0}^{\varepsilon} dr \left|\int_{\mathbb{R}^{q-1}} \prod_{i=1}^{q-1} d B_{\xi_i} \int_{0}^{\varepsilon}d \theta\left(\left(\theta - r\right)_{+}^{\frac{H-1}{q}-\frac{1}{2}}\prod_{j=1}^{q-1}\left(\theta + h -\xi_j\right)_{+}^{\frac{H-1}{q}-\frac{1}{2}}\right)\right|^2\nonumber\\
		&= q^2 c^2(H, q) \, \int_{0}^{\varepsilon}dr\,\left|\int_{\mathbb{R}^{q-1}} \prod_{i=1}^{q-1}d \tilde{B}_{\xi_i}^{(h)} \int_{0}^{\varepsilon}d \theta\left(\left(\theta - r\right)_{+}^{\frac{H-1}{q}-\frac{1}{2}}\prod_{j=1}^{q-1}\left(\theta -\xi_j\right)_{+}^{\frac{H-1}{q}-\frac{1}{2}}\right)\right|^2 \nonumber\\
	\end{align}
	where $\tilde{B}^{(h)}=\{B_{t+h}-B_h\,,\;t\in \R\}$. Since $\tilde{B}^{(h)}$ is itself a two-sided Brownian motion, the right-hand side of \eqref{eq:ZZZt} is equal in distribution to
	\begin{align}\label{eq:ZZZt2}
		&q^2 c^2(H, q) \, \int_{0}^{\varepsilon}dr\,\left|\int_{\mathbb{R}^{q-1}} \prod_{i=1}^{q-1}d {B}_{\xi_i} \int_{0}^{\varepsilon}d \theta\left(\left(\theta - r\right)_{+}^{\frac{H-1}{q}-\frac{1}{2}}\prod_{j=1}^{q-1}\left(\theta -\xi_j\right)_{+}^{\frac{H-1}{q}-\frac{1}{2}}\right)\right|^2
	\end{align}
	Applying the changes of variables $r=\varepsilon s$ and $\theta = \varepsilon\, \tau$, \eqref{eq:ZZZt2} rewrites as

	\begin{align}\label{eq:ZZZt3}
		&q^2 c^2(H, q) \, \varepsilon^{\frac{2(H-1)}{q}+ 2}\, \int_{0}^{1}ds\, \nonumber\\
		&\hskip1cm\left|\int_{\mathbb{R}^{q-1}} \prod_{i=1}^{q-1}d {B}_{\xi_i} \int_{0}^{1}d \tau\left(\left(\tau - s\right)_{+}^{\frac{H-1}{q}-\frac{1}{2}}\prod_{j=1}^{q-1}\left(\varepsilon\, \tau -\xi_j\right)_{+}^{\frac{H-1}{q}-\frac{1}{2}}\right)\right|^2 \nonumber\\
		&=q^2 c^2(H, q) \, \varepsilon^{\frac{2(H-1)}{q}+ 2}\, \int_{0}^{1}ds \nonumber\\
		&\hskip1cm \left|\int_{\mathbb{R}^{q-1}} \prod_{i=1}^{q-1}d {B}_{\varepsilon \,\xi_i} \int_{0}^{1}d \tau\left(\left(\tau - s\right)_{+}^{\frac{H-1}{q}-\frac{1}{2}}\prod_{j=1}^{q-1}\left(\varepsilon\, \tau -\varepsilon\,\xi_j\right)_{+}^{\frac{H-1}{q}-\frac{1}{2}}\right)\right|^2 \nonumber\\
		&=q^2 c^2(H, q) \, \varepsilon^{2H -(q-1)}\, \int_{0}^{1}ds \nonumber\\
		&\hskip1cm \left|\int_{\mathbb{R}^{q-1}} \prod_{i=1}^{q-1}d {B}_{\varepsilon \,\xi_i} \int_{0}^{1}d \tau\left(\left(\tau - s\right)_{+}^{\frac{H-1}{q}-\frac{1}{2}}\prod_{j=1}^{q-1}\left(\tau -\xi_j\right)_{+}^{\frac{H-1}{q}-\frac{1}{2}}\right)\right|^2 
	\end{align}
	Finally, by the scaling property of Brownian motion, the right-hand side of \eqref{eq:ZZZt3} is equal in distribution to
	\begin{align*}
		& \varepsilon^{2H}\, q^2 c^2(H, q) \,  \int_{0}^{1}ds \nonumber\\
		&\hskip1cm \left|\int_{\mathbb{R}^{q-1}} \prod_{i=1}^{q-1}d {B}_{\xi_i} \int_{0}^{1}d \tau\left(\left(\tau - s\right)_{+}^{\frac{H-1}{q}-\frac{1}{2}}\prod_{j=1}^{q-1}\left(\tau -\xi_j\right)_{+}^{\frac{H-1}{q}-\frac{1}{2}}\right)\right|^2 \nonumber\\
		&= \varepsilon^{2H}\, \int_{0}^{1} dr|D_rZ^{H,q}_{1}|^2.
	\end{align*}
	This completes the proof.
\end{proof}

\begin{Thm}\label{thm:HHermite}
	Let $Z^{H,q}=(Z^{H,q,1},\ldots,Z^{H,q,m})$ be the $m$-dimensional Hermite process of order $q\geq 1$ and self-similarity parameter $H\in \left(\frac{1}{2},1\right)$ defined in \eqref{eqn:defofmultihermitekernel}. Then $Z^{H,q}$ satisfies hypothesis $\mathbf{(H_5)}$.
\end{Thm}

\begin{proof}
Let $t\in (0,T]$ and let $Y=(Y_r)_{r\in [0,t]}$ be a real-valued stochastic process with $\gamma$-Hölder continuous trajectories, for some $\gamma>\frac{1}{2}$. It suffices to show that, given a one-dimensional Hermite process $Z=\{Z_s\}_{s\in \R_+}$ of order $q\geq 1$ and self-similarity parameter $H\in (\frac{1}{2},1)$, the existence of some $\Omega_0\in \mathcal{F}$ with $\P(\Omega_0)>0$ satisfying
	\begin{align}\label{eq:ZeroNorm1}
		\left\|\int_{0}^{t} Y_s(\omega)dDZ_s(\omega)\right\|_{L^2(\R),\R}=0 \qquad \text{for every } \omega\in\Omega_0
	\end{align}
implies the existence of $\Omega_1\in \mathcal{F}$ with $\P(\Omega_1)>0$ such that
	\begin{align}\label{eq:ClaimY}
		Y_t(\omega)=0 \qquad \text{for every } \omega\in \Omega_1.
	\end{align}
	Fix $\varepsilon\in (0,t)$ and set $h=t-\varepsilon$. By \eqref{eq:ZeroNorm1}, we have, for every $\omega\in\Omega_0$,
	\begin{align}\label{eq:ZeroNormst}
		0=\left\|\int_{0}^{h}  Y_s(\omega) dDZ_{s}(\omega) + \int_{h}^{\varepsilon + h}  Y_s(\omega) dDZ_{s}(\omega)\right\|_{L^2([h,h+\varepsilon], \R)}^2.
	\end{align}
	Observe now that, by \eqref{eq:adaptfromnualart}, for any $\varepsilon\in (0,t)$ and $s\in [0,h]$ one has $D_rZ_s(\omega)=0$ for $\lambda\otimes \P$-a.e.\ $(r,\omega)\in[h,h+\varepsilon]\times\Omega$. Consequently, there exists $\Omega'_0\in \mathcal{F}$ with $\P(\Omega'_0)>0$ such that, for every $\varepsilon\in (0,t)$ and $\omega\in\Omega'_0$,
	\begin{align}\label{eq:ZeroNorms}
		\int_{0}^{h}  Y_s(\omega) dDZ_{s}(\omega)=0 \qquad \text{in } L^2([h,h+\varepsilon], \R).
	\end{align}
	Combining \eqref{eq:ZeroNormst} and \eqref{eq:ZeroNorms}, we obtain, for every $\varepsilon\in(0,t)$ and $\omega\in\Omega'_0$,
	\begin{align}\label{eq:ZeroNorms2}
		0&=\left\|\int_{h}^{\varepsilon+ h}  Y_s(\omega) dDZ_{s}(\omega)\right\|_{L^2([h,h+\varepsilon], \R)}^2.
	\end{align}
	Applying the Sewing lemma \cite[Proposition 2.1]{li2022mild} together with \eqref{eq:ZeroNorms2} yields, for every $\varepsilon\in(0,t)$ and $\omega\in\Omega'_0$,
	\begin{align*}
		&\left\|Y_h(\omega) \big(DZ_{\varepsilon+h}(\omega)-DZ_{h}(\omega)\big)\right\|_{L^2([h,h+\varepsilon], \R)} \leq C_{\omega}\varepsilon^{\gamma + H -\beta},
	\end{align*}
where $\beta\in (0,H\wedge \alpha)$. We thereby infer that, for every $\varepsilon\in(0,t)$ and $\omega\in\Omega'_0$,
	\begin{align}\label{eq:YY}
		|Y_h(\omega)| &\leq \frac{C_{\omega}\,\varepsilon^{\gamma + H -\beta}}{\left\| DZ_{\varepsilon+h}(\omega)-DZ_{h}(\omega)\right\|_{L^2([h,h+\varepsilon], \R)}}.
	\end{align}
	By Lemma \ref{Prop:Zt9999}, we have, for every $\varepsilon\in (0,t)$,
	\begin{align}\label{eq:LawDZ}
		\frac{\varepsilon^{\gamma + H -\beta}}{\left\| DZ_{\varepsilon+h}-DZ_{h}\right\|_{L^2([h,h+\varepsilon], \R)}}\overset{\text{law}}{=}	\frac{\varepsilon^{\gamma -\beta}}{\left\| DZ_{1}\right\|_{L^2([0,1])}}.
	\end{align}
	Since $\gamma - \beta>0$, the right-hand side of \eqref{eq:LawDZ} converges in probability to $0$ as $\varepsilon\to 0^+$, and therefore so does the left-hand side. Hence, there exists a sequence $(\varepsilon_n)_{n\geq 1}$ tending to $0^+$ such that, for almost every $\omega\in \Omega$,
	\begin{align}\label{eq:LimitEpsilon}
		\lim_{n\to \infty}\frac{\varepsilon_n^{\gamma + H -\beta}}{\left\| DZ_{t}(\omega)-DZ_{t-\varepsilon_n}(\omega)\right\|_{L^2([t-\varepsilon_n,t], \R)}}=0.
	\end{align}
	Combining \eqref{eq:YY} and \eqref{eq:LimitEpsilon}, we conclude that there exists $\Omega_1\in \mathcal{F}$ with $\P(\Omega_1)>0$ such that $Y_t(\omega)=0$ for almost every $\omega\in \Omega_1$, which proves \eqref{eq:ClaimY}.
\end{proof}

\appendix

\section{Proof of Lemma \ref{indepnt}} \label{appendix:proofofLemma}
Let $(\mathfrak H,\langle\cdot,\cdot\rangle_{\mathfrak H})$ be a real separable Hilbert space, and fix $e_0\in\mathfrak H$ with $\|e_0\|_{\HHH}=1$. Let $(e_n)_{n\geq 1}$ be an orthonormal basis of the orthogonal complement of $\operatorname{span}\{e_0\}$ in $\mathfrak H$, so that $(e_n)_{n\geq 0}$ is an orthonormal basis of $\HHH$. Every $f\in \HHH^{\odot  q}$ then admits the expansion
\begin{equation}\label{eq:fg-expansion}
f = \sum_{j_1,\ldots,j_q=0}^{\infty}\lambda(f;j_1,\ldots,j_q)\, e_{j_1}\otimes\cdots\otimes e_{j_q},
\end{equation}
with coefficients
\[
\lambda(f;j_1,\ldots,j_q)
=
\big\langle f,\, e_{j_1}\otimes\cdots\otimes e_{j_q}\big\rangle_{\mathfrak H^{\otimes q}}.
\]
Since $f$ is symmetric, $\lambda$ is invariant under permutations of its arguments: for every permutation $\sigma$ of $\{1,\dots,q\}$,
$$
\lambda(f;j_{\sigma(1)},\ldots,j_{\sigma(q)})=\lambda(f;j_1,\ldots,j_q).
$$

For notational convenience, given $e_{j_1},\ldots, e_{j_q}\in \HHH$ we set
\begin{align*}
	e_{j_1}\odot\cdots\odot e_{j_q} &:= \frac{1}{q!}\sum_{\sigma} e_{j_{\sigma(1)}}\otimes\cdots\otimes e_{j_{\sigma(q)}},
\end{align*}
where the sum runs over all permutations $\sigma$ of $\{1,\dots,q\}$. We shall also use the shorthand
\begin{align*}
	e_0^{\odot  k} \odot e_{j_1} \odot \ldots \odot e_{j_{q-k}}&:= 	\underbrace{e_0\odot \cdots \odot e_0}_{k\text{ times}} \odot e_{j_1} \odot \ldots \odot e_{j_{q-k}}.
\end{align*}
Equivalently (see, e.g., \cite[Eq.\ (B.3.1), p.~206]{nourdin2012normal}), the expansion \eqref{eq:fg-expansion} can be rewritten in terms of symmetric tensors as
\begin{equation}\label{eq:fg-sym-expansion}
f = \sum_{j_1,\ldots,j_q=0}^{\infty}\lambda(f;j_1,\ldots,j_q)\, e_{j_1}\odot\cdots\odot e_{j_q}.
\end{equation}
Grouping the terms according to the number of indices equal to $0$, we may rewrite \eqref{eq:fg-sym-expansion} as
\begin{equation}\label{eq:fg-sym-expansion2}
	f=\sum_{k=0}^q \sum_{j_1, \ldots, j_{q-k}=0}^{\infty} a_k^{(q)}\left(f;j_1, \ldots, j_{q-k}\right) e_0^{\odot  k} \odot e_{j_1} \odot \ldots \odot e_{j_{q-k}},
\end{equation}
where
\begin{align*}
	a_k^{(q)}\left(f;j_1, \ldots, j_{q-k}\right) &:= \lambda(f;j_1,\ldots,j_{q-k}, \underbrace{0,\ldots,0}_{k\text{ times}}).
\end{align*}

\medskip

Before turning to the proof of Lemma \ref{indepnt}, we recall the notion of contraction of tensors. Fix $r\in\{0,1,\dots,m\wedge n\}$. We first define the $r$th contraction on elementary tensors. For $r\geq 1$, we set
\[
\begin{aligned}
&(e_{j_1}\otimes\cdots\otimes e_{j_n})\otimes_r (e_{k_1}\otimes\cdots\otimes e_{k_m}) \\
&\qquad :=
\Bigg(\prod_{\ell=1}^r \langle e_{j_\ell},e_{k_\ell}\rangle_{\mathfrak H}\Bigg)\,
e_{j_{r+1}}\otimes\cdots\otimes e_{j_n}\otimes e_{k_{r+1}}\otimes\cdots\otimes e_{k_m}
\;\in\;\mathfrak H^{\otimes(m+n-2r)}.
\end{aligned}
\]
For $r=0$, we use the convention
\[
(e_{j_1}\otimes\cdots\otimes e_{j_n})\otimes_0 (e_{k_1}\otimes\cdots\otimes e_{k_m})
=
e_{j_1}\otimes\cdots\otimes e_{j_n}\otimes e_{k_1}\otimes\cdots\otimes e_{k_m}.
\]
In particular, when $m=n=r$,
\[
(e_{j_1}\otimes\cdots\otimes e_{j_n})\otimes_n (e_{k_1}\otimes\cdots\otimes e_{k_n})
=\prod_{\ell=1}^n \langle e_{j_\ell},e_{k_\ell}\rangle_{\mathfrak H}
=
\big\langle e_{j_1}\otimes\cdots\otimes e_{j_n},\,e_{k_1}\otimes\cdots\otimes e_{k_n}\big\rangle_{\mathfrak H^{\otimes n}}.
\]

The definition is then extended to general tensors by bilinearity and continuity. If $f\in\mathfrak H^{\otimes n}$ and $g\in\mathfrak H^{\otimes m}$ admit expansions as in \eqref{eq:fg-expansion}, their $r$th contraction is given by
\begin{equation}\label{eq:contraction-general}
\begin{aligned}
f\otimes_r g
&:= \sum_{j_1,\ldots,j_n\ge1}\sum_{k_1,\ldots,k_m\ge1}
\lambda(f;j_1,\ldots,j_n)\,\lambda(g;k_1,\ldots,k_m)\\
&\hskip3cm\times
(e_{j_1}\otimes\cdots\otimes e_{j_n})\otimes r (e_{k_1}\otimes\cdots\otimes e_{k_m}).
\end{aligned}
\end{equation}
By construction, $f\otimes_r g\in \mathfrak H^{\otimes(m+n-2r)}$ for every $r\le m\wedge n$. In general, $f\otimes_r g$ is not symmetric, and we denote by
\[
f\widetilde\otimes_r g \in \mathfrak H^{\odot(m+n-2r)}
\]
its symmetrization.

In the canonical case $\mathfrak H=L^2(T,\mathcal B,\mu)$, with $\mu$ $\sigma$-finite and non-atomic, formula \eqref{eq:contraction-general} takes the form, for $\mu^{m+n-2r}$-a.e.\ $(t_1,\dots,t_{m+n-2r})\in T^{m+n-2r}$,
\begin{equation}\label{eq:contraction-L2}
\begin{aligned}
&(f\otimes_r g)(t_1,\ldots,t_{m+n-2r})\\
&=
\int_{T^r} f(u_1,\ldots,u_r,t_1,\ldots,t_{n-r})\,\\
&\hskip3cm\times 
g(u_1,\ldots,u_r,t_{n-r+1},\ldots,t_{m+n-2r})\,
\mu^{\otimes r}(d u).
\end{aligned}
\end{equation}

We are now in a position to prove Lemma \ref{indepnt}, which is a key ingredient in the proof of Theorem \ref{thm:ray}, and hence also in the proof of Theorem \ref{thm:MallDiff}.

\begin{proof}[Proof of Lemma \ref{indepnt}]
	Choose an orthonormal basis $(e_n)_{n\geq 1}$ of $V$. By \eqref{eq:fg-sym-expansion2}, we have the expansion
	\begin{equation}\label{eq:fg-sym-expansion3}
		f_t=\sum_{k=0}^q \sum_{j_1, \ldots, j_{q-k}=0}^{\infty} a_{k}^{(q)}\left(f_t;j_1, \ldots, j_{q-k}\right) e_0^{\odot  k} \odot e_{j_1} \odot \ldots \odot e_{j_{q-k}},
	\end{equation}
where
\begin{align*}
	a_{k}^{(q)}\left(f_t;j_1, \ldots, j_{q-k}\right) &:=\lambda(f_t;j_1,\ldots,j_{q-k}, \underbrace{0,\ldots,0}_{k\text{ times}}).
\end{align*}
Since $f_t\in \HHH^{\odot  q}$, we have, for every $k=0,\ldots,q$,
$$ \sum_{j_1, \ldots, j_{q-k}=0}^{\infty} \left|a_{k}^{(q)}\left(f_t;j_1, \ldots, j_{q-k}\right)\right|^2<\infty.$$
For each $k=0,\ldots,q$, we may therefore define the vector $f_{t,q-k}\in \HHH^{\odot  {q-k}}$ by
\begin{align*}
	f_{t,q-k} &:= \sum_{j_1, \ldots, j_{q-k}=0}^{\infty} a_{k}^{(q)}\left(f_t;j_1, \ldots, j_{q-k}\right) e_{j_1} \odot \ldots \odot e_{j_{q-k}}.
\end{align*}
Plugging this definition into \eqref{eq:fg-sym-expansion3}, we obtain the tensor decomposition in $\mathfrak{H}^{\odot  q}$
\begin{equation}
	f_t=\sum_{k=0}^q e_0^{\otimes k} \odot f_{t,q-k}.
\end{equation}
Moreover,
\begin{equation}\label{eq:fg-sym-expansion4}
	\|f_t-f_s\|_{\HHH^{\otimes q}}=\sum_{k=0}^q \|e_0^{\otimes k} \odot (f_{t,q-k}-f_{s,q-k})\|_{\HHH^{\otimes q}},
\end{equation}
from which the inequality \eqref{eq:ftn-k-holder} follows directly.

It remains to establish the decomposition \eqref{indep}. To this end, observe that
$$
I_q(f_t)=\sum_{k=0}^q I_q\left(e_0^{\otimes k} \odot f_{t,q-k}\right) \stackrel{(\star)}{=} \sum_{k=0}^q I_k\left(e_0^{\otimes k}\right) I_{q-k}\left(f_{t,q-k}\right),
$$
where $(\star)$ follows from the product formula \eqref{eq:product-formula}, combined with the vanishing of the contractions $e_0^{\otimes k} \otimes r f_{t,q-k}$ for $k=1, \ldots, q-1$ and $r=1, \ldots, k \wedge(q-k)$. This concludes the proof.
\end{proof}

\bigskip
 
 \noindent
 {\bf Acknowledgements}.
L. L.  acknowledges support from the FRNS Research Credit grant J.0136.26 ``NOISE''. 

Part of this work was carried out while Y. N. was at the University of Luxembourg under a grant funding from the European Union’s Horizon 2020 research and innovation programme N° 811017.

I.N. gratefully acknowledges support from the Luxembourg National Research Fund
(Grant No. O22/17372844/FraMStA).

\end{document}